\documentclass[10pt,twoside,a4paper]{article}
\usepackage{amsfonts,amssymb,amsmath,amscd,latexsym,makeidx,theorem}
\author{Luca Martinazzi\thanks{The first author was supported by the ETH Research 
Grant no. ETH-02 08-2 and by the Italian FIRB Ideas ``Analysis and beyond''.}\\ 
\small{Centro De Giorgi, Pisa} \\ \footnotesize{\texttt{luca.martinazzi@sns.it}}
\and Michael Struwe\\ \small{ETH Zurich} \\ \footnotesize{\texttt{struwe@math.ethz.ch}}}
\title{Quantization for an elliptic equation of order $2m$ with critical exponential non-linearity}

\newtheorem{trm}{Theorem}
\newtheorem{prop}[trm]{Proposition}

\newtheorem{lemma}[trm]{Lemma}

\newcommand{\R}[1]{\mathbb{R}^{#1}}
\newcommand{\N}{\mathbb{N}}
\newcommand{\de}{\partial}
\newcommand{\ve}{\varepsilon}

\newenvironment{proof}{\noindent\emph{Proof.}}{\hfill$\square$\medskip}

\DeclareMathOperator{\loc}{loc}

\DeclareMathOperator*{\dist}{dist}

\DeclareMathOperator*{\Intm}{\int\!\!\!\!\!\! \rule[2.6pt]{6.5pt}{.4pt}}

\DeclareMathOperator{\vol}{vol}

\begin{document}

\maketitle

\begin{abstract}
On a smoothly bounded domain $\Omega\subset\R{2m}$ we consider a sequence of positive 
solutions $u_k\stackrel{w}{\rightharpoondown} 0$ in $H^m(\Omega)$ to the equation 
$(-\Delta)^m u_k=\lambda_k u_k e^{mu_k^2}$ subject to 
Dirichlet boundary conditions, where $0<\lambda_k\to 0$. Assuming that
$$\Lambda:=\lim_{k\to\infty}\int_\Omega u_k(-\Delta)^m u_k dx<\infty,$$
we prove that $\Lambda$ is an integer multiple of $\Lambda_1:=(2m-1)!\vol(S^{2m})$,
the total $Q$-curvature of the standard $2m$-dimensional sphere.
\end{abstract}

\section{Introduction}
Given a smoothly bounded domain $\Omega\subset\R{2m}$, suppose that for each $k \in \N$ 
we have a smooth function $u_k > 0$ satisfying the equation 
\begin{equation}\label{eq0}
  (-\Delta)^m u_k =\lambda_k u_k e^{m u_k^2} \textrm{ in } \Omega
\end{equation}
with 
\begin{equation}\label{1.2}
  u_k=\partial_\nu u_k=\ldots =\partial_\nu^{m-1} u_k=0 \textrm{ on }\partial\Omega,
\end{equation}
where $0<\lambda_k\to 0$ as $k \to \infty$. We assume that $(u_k)$ is bounded in $H^m(\Omega)$.
Hence, after passing to a subsequence and integrating by parts we may assume that
as $k \to \infty$ we have
\begin{equation}\label{eq1}
  \int_\Omega |\nabla^m u_k|^2 dx = \int_\Omega u_k(-\Delta)^m u_k dx
  =\lambda_k\int_\Omega u_k^2 e^{m u_k^2}dx\to\Lambda >0.
\end{equation}
Note that by elliptic estimates the quantity
$$\|u\|:=\big(\int_\Omega |\nabla^m u_k|^2 dx\big)^{1/2} 
=\big(\int_\Omega \sum_{|\alpha| = m} |\partial^{\alpha} u_k|^2 dx\big)^{1/2} $$
defines a norm on the Beppo-Levi space $H^m_0(\Omega)$ which is equivalent to the standard 
Sobolev norm. 

Generalising previous results by Adimurthi and Struwe \cite{AS}, Adimurthi and Druet \cite{AD} 
and Robert and Struwe \cite{RS}, the first author proved in \cite{mar2} the following theorem. 

\begin{trm}\label{trm1} Let $(u_k)$ be a sequence of positive solutions to \eqref{eq0}, \eqref{1.2}
with $0<\lambda_k\to 0$ as $k \to \infty$ and satisfying \eqref{eq1} for some $\Lambda>0$. 
Then $\sup_\Omega u_k\to \infty$ as $k\to\infty$ 
and there exist a subsequence $(u_k)$ and sequences of points 
$x_k^{(i)}\to x^{(i)}\in \overline\Omega$, $1 \le i \le I$, for some integer $I\leq C\Lambda$,
such that the following is true.

For every $1\leq i\leq I$, letting $r_k^{(i)}>0$ be given by
\begin{equation}\label{rik}
\lambda_k (r_k^{(i)})^{2m} u_k^2(x_k^{(i)}) e^{m u_k^2(x_k^{(i)})}=2^{2m}(2m-1)!
\end{equation}
and setting 
$$\eta_k^{(i)}(x):=u_k(x_k^{(i)})(u_k(x_k^{(i)}+r_k^{(i)}x)-u_k(x_k^{(i)}))+\log 2,$$
we have $r_k^{(i)}\to 0$, $\dist(x_k^{(i)},\de \Omega)/r_k^{(i)}\to \infty$ 
as $k\to\infty$, and
\begin{equation}\label{etak}
\eta_k^{(i)}(x) \to \eta_0(x) =\log\frac{2}{1+|x|^2}\textrm{ in } C^{2m-1}_{\loc}(\R{2m})
\textrm{ as } k\to\infty.
\end{equation}
Moreover, for $i \neq j$ there holds
\begin{equation}\label{1.6} 
  \frac{|x_k^{(i)}-x_k^{(j)}|}{r_k^{(i)}}\to\infty \textrm{ as } k\to\infty.
\end{equation}

In addition, with $R_k(x):=\inf_{1\leq i\leq I}|x-x_k^{(i)}|$ there exists a constant $C > 0$
such that there holds
\begin{equation}\label{eq9}
 \lambda_k R_k^{2m}(x) u_k^2(x)e^{m u_k^2(x)} \leq C
\end{equation}
uniformly for all $x\in \Omega$, $k \in \N$.

Finally $u_k\to 0$ in $C^{2m-1}_{\loc}(\overline \Omega\backslash S)$, 
where $S=\{x^{(1)},\ldots, x^{(I)}\}$.
\end{trm}

We remark that the function $\eta_0$ given by \eqref{etak} satisfies the $Q$-curvature equation
\begin{equation}\label{eqliou}
(-\Delta)^m \eta_0 =(2m-1)!e^{2m\eta_0}
\end{equation}
and
\begin{equation}\label{etak2}
(2m-1)!\int_{\R{2m}}e^{2m \eta_0}dx=\int_{S^{2m}}Q_{S^{2m}}\mathrm{dvol}_{g_{S^{2m}}}
=(2m-1)!|S^{2m}|=:\Lambda_1.
\end{equation}
For a discussion of the geometric meaning of \eqref{eqliou} we refer to \cite{cha} 
or to the introduction of \cite{mar1}.

\medskip

The purpose of this paper is to prove the following quantization result.

\begin{trm}\label{trm2} Under the hypothesis of Theorem \ref{trm1} we have $\Lambda= I^*\Lambda_1$ 
for some $I^*\in \mathbb{N}\backslash \{0\}$.
\end{trm}

The analogue of Theorem \ref{trm2} was proven by O. Druet \cite{dru} in dimension $2$ ($m=1$) 
and by the second author \cite{str} in dimension $4$ ($m=2$) in the case of the Navier boundary
condition $u_k=\Delta u_k=0$ on $\de \Omega$.
Note that in the latter case the maximum principle implies that $\Delta u_k\leq 0$ in $\Omega$ whereas
such an estimate is not available in the case of the Dirichlet boundary condition. 

Quantization results similar to Theorem \ref{trm2} previously have also been obtained for concentrating 
sequences of solutions $u_k$ to the $Q$-curvature equation 
\begin{equation}\label{eq2}
(-\Delta)^m u_k=\lambda_k e^{2mu_k}\quad \text{in }\Omega\subset\R{2m}.
\end{equation}
In the case of the Navier boundary condition, assuming that $\lambda_k\to 0$ and
$$\Lambda:=\lim_{k\to\infty}\int_\Omega \lambda_k e^{2mu_k}dx <\infty,$$
J. Wei \cite{wei} proved that when $m=2$ and when $\Omega$ is convex the quantity $\Lambda$ 
is an integer multiple of $\Lambda_1$. Moreover concentration points are simple and isolated, 
in the sense that $x^{(i)}\neq x^{(j)}$ for $i\neq j$, and $I^*=I$ in the notation of 
Theorems \ref{trm1} and \ref{trm2} above. Robert and Wei \cite{RW} proved the analogous result
for a general domain $\Omega$ and in the case of Dirichlet boundary conditions.
In \cite{MP}, the first author and Petrache generalized the result of Robert and Wei to 
arbitrary dimensions.

Equation \eqref{eq0} is more difficult to deal with analytically than equation \eqref{eq2};
the analogous questions whether for a blowing up sequence of solutions to \eqref{eq0} 
the concentration points are isolated, simple and stay away from the boundary 
are still open, even in dimension $2$.

\medskip

Our paper is organized as follows. In the next section we present the proof of Theorem \ref{trm2}
in the case when $\Omega=B_R$ is a ball and each function $u_k$ is radially symmetric. 
In Section \ref{sec3} we prove the theorem in the general case.
Some useful technical results are collected in the Appendix. The overall strategy of the proof 
is very similar to the approach followed in \cite{str}, and some of the 
results in \cite{str} can be carried over almost literally to the present setting. 
Several key steps in the proof, however, require conceptually new ideas in the case when $m \ge 3$.
These ideas also shed new light on the previous approaches in low dimensions and have a unifying 
feature. 

Throughout the paper the letter $C$ denotes a generic 
constant independent of $k$ which can change from line to line, or even within the same line. 

\section{Proof of Theorem \ref{trm2} in the radial case}\label{sec2}

Let $\Omega=B_R=B_R(0)$ and assume that each $u_k$ is radially symmetric. By slight abuse of 
notation we write $u_k(x)=u_k(r)$ if $|x|=r$. In the notation of Theorem \ref{trm1} we then
have $I=1$ and we can choose $x_k^{(1)}=0$ for every $k>0$. In fact, as shown in assertion 
\eqref{stima0.5} of Lemma \ref{lemma2} below, we have $u_k(0)=\max_{\Omega} u_k$.

\subsection{Strategy of the proof}

Set $e_k:=\lambda_k u_k^2 e^{mu_k^2}$ and let 
$$\Lambda_k(r):=\int_{B_r}e_kdx,\quad 
   N_k(s,t):=\Lambda_k(t)-\Lambda_k(s)=\int_{B_t\backslash B_s}e_k dx$$
as in \cite{str}. We shall say that the property $(H_\ell)$ is satisfied if there exist 
sequences
$$s_k^{(0)}:=0<r_k^{(1)}< s_k^{(1)}<\ldots<r_k^{(\ell)}<s_k^{(\ell)}\leq R,\quad k\in\mathbb{N},$$
such that the following holds:
\begin{itemize}
\item[$(H_{\ell,1})$] $\lim_{k\to\infty}\frac{r_k^{(j)}}{s_k^{(j)}}
=\lim_{k\to\infty}\frac{s_k^{(j-1)}}{r_k^{(j)}}=0$ for $1\leq j\leq \ell$,
\item[$(H_{\ell,2})$] $\lim_{k\to\infty}\frac{u_k(s_k^{(j)})}{u_k(Lr_k^{(j)})}=0$ 
for $1\leq j\leq \ell$, $L>0$,
\item[$(H_{\ell,3})$] $\lim_{k\to\infty}\Lambda_k(s_k^{(j)})= j\Lambda_1$ for $1\leq j\leq \ell$,
\item[$(H_{\ell,4})$] $\lim_{L\to\infty}\lim_{k\to\infty}
\big(N_k(s_k^{(j-1)},r_k^{(j)}/L)+ N_k(Lr_k^{(j)},s_k^{(j)})\big)=0$ 
for $1\leq j\leq \ell$.
\end{itemize}

For the proof of Theorem \ref{trm2} we proceed via induction from the following 
two claims: $(H_1)$ holds, and if $(H_\ell)$ holds then either $(H_{\ell+1})$ holds as well, or
\begin{equation}\label{Nk0}
\lim_{k\to\infty} N_k(s_k^{(\ell)},R)=0.
\end{equation}
By \eqref{eq1} and $(H_{\ell,3})$ the induction terminates when $\ell>\frac{\Lambda}{\Lambda_1}$.
Letting ${\ell_0}$ be the largest integer such that $(H_{\ell_0})$ holds, $(H_{{\ell_0},3})$ and \eqref{Nk0} imply
$$\Lambda=\lim_{k\to\infty}\Lambda_k(s_k^{({\ell_0})})+\lim_{k\to\infty}N_k(s_k^{({\ell_0})},R)
={\ell_0}\Lambda_1,$$
and Theorem \ref{trm2} in the radial case follows.

\subsection{Proof of $(H_1)$}
Let $r_k>0$ be defined as in Theorem \ref{trm1} such that
$$\lambda_k r_k^{2m} u_k^2(0) e^{m u_k^2(0)}=2^{2m}(2m-1)!, $$
and set 
$$w_k(x):=u_k(0)(u_k(x)-u_k(0))\ \textrm{ in }B_R.$$  
We have
\begin{eqnarray*}
(-\Delta)^m w_k&=&\lambda_k u_k(0)u_k e^{m u_k^2}\\
&=&\lambda_k u_k(0)u_ke^{mu_k^2(0)}e^{2m\big(1+\frac{w_k}{2u_k^2(0)}\big)w_k}=:f_k\ \textrm{ in }B_R.
\end{eqnarray*}
Letting also
$$\sigma_k(r):=\int_{B_r}f_kdx\geq \Lambda_k(r), $$
then by \eqref{etak} of Theorem \ref{trm1} and \eqref{etak2} clearly we have
\begin{eqnarray}
\lim_{L\to\infty}\lim_{k\to\infty}\Lambda_k(Lr_k)
&=&\lim_{L\to\infty}\lim_{k\to\infty}\sigma_k(Lr_k)\nonumber\\
&=& \lim_{L\to\infty}\lim_{k\to\infty} (2m-1)!\int_{B_L}e^{2m\eta_k}dx=\Lambda_1.\label{etak3}
\end{eqnarray}
For $0<t\leq R$ let $g_k$ solve the equation 
$$ \Delta^m g_k=\Delta^m w_k\ \textrm{ in } B_t$$
with homogeneous Dirichlet boundary data
$$g_k=\de_\nu g_k=\ldots=\de_\nu^{m-1} g_k=0\ \text{ on }\de B_t.$$
Then Lemma \ref{lemmaparti} in the Appendix gives the identity
\begin{equation}\label{parti5}
(-1)^m\de_\nu^m g_k(t)=\frac{A_k(t)}{\omega_{2m-1}t^{3m-2}}
\end{equation}
similar to (20) in \cite{str}, where
\begin{equation}\label{parti6}
A_k(t):=\int_0^{t}t_2\cdots\int_0^{t_{m-1}}t_m \sigma_k(t_m) dt_m\ldots dt_2,
\end{equation}
and where $\omega_{2m-1}$ is the $(2m-1)$-dimensional volume of $S^{2m-1}$. 

\begin{lemma}\label{lemmaL}
For every $b<2$ we can find $L=L(b)$ and $k_0=k_0(b)$ such that for $k\geq k_0$ we have
\begin{equation}\label{denu}
(-\de_\nu)^mg_k(t) \geq \frac{2^{m-1}(m-1)!b}{t^m} \ \text{ for } Lr_k\le t \le R.
\end{equation}
\end{lemma}

\begin{proof}
Noting that 
$$\frac{\Lambda_1}{\omega_{2m-1}2^{m-1}(m-1)!}=2^m (m-1)!,$$
from \eqref{etak3} and \eqref{parti5} together with the identity
$$\int_0^{t}t_2\cdots\int_0^{t_{m-1}}t_mdt_m\cdots dt_2=\frac{t^{2m-2}}{2^{m-1}(m-1)!}$$ 
we obtain the claim.
\end{proof}

These estimates now yield the following result analogous to Lemma 2.1 in \cite{str}.
Note, however, that the statement \eqref{stima0.5} below in the present case no longer can simply
be deduced from the maximum principle, as was the case in \cite{str}. In addition, the higher
order nature of equation \eqref{eq0} requires substantial technical modifications of the approach
used in \cite{str}.   

\begin{lemma}\label{lemma2} For any $b<2$ there is $L=L(b)$ and $k_0=k_0(b)$ such that 
for $k\geq k_0$ there holds
\begin{eqnarray}
w_k'(t)&\leq& -\frac{b}{t}+tP(t) \quad \textrm{in }B_R\backslash B_{Lr_k},\label{stima0}\\
w_k'(t)&\leq& 0 \quad \textrm{in }B_R,\label{stima0.5}\\
w_k(t)&\leq& b\log\Big(\frac{r_k}{t}\Big)+C \quad \textrm{in }B_R,\label{stima1}
\end{eqnarray}
where $P$ is a polynomial independent of $k$. In particular $u_k$ is monotone decreasing. 
For any $\ve \in ]0,1[$ let $T_k>0$ be such that $u_k(T_k)=\ve u_k(0)$. Then we have
\begin{equation}\label{stima1c}
\lim_{k\to\infty}\frac{r_k}{T_k}=0
\end{equation}
and
\begin{equation}\label{stima1b}
\lim_{k\to\infty} \Lambda_k(T_k)=\lim_{k\to\infty} \sigma_k(T_k)=\Lambda_1.
\end{equation}
\end{lemma}

\begin{proof} Fix $t>0$ and write $w_k=g_k+h_k$, where
$$\Delta^m h_k=0\textrm{ in }B_t,\ 
\textrm{ and } g_k=\partial_\nu g_k=\ldots=\partial_\nu^{m-1}g_k=0\ \textrm{ on }\partial B_t.$$

\medskip

\noindent\emph{Step 1.} We claim that
\begin{equation}\label{ratta}
\partial_\nu^m g_k(t)=t^{m-1}\underbrace{(t^{-1} (t^{-1}\cdots (t^{-1}}_{m-1 
\textrm{ times}}w_k'(t)\underbrace{)'\cdots)')'}_{m-1\textrm{ times}}.
\end{equation}
Indeed, subtracting $\partial_\nu^m w_k(t)$ from both sides of \eqref{ratta} we need to show
$$-\partial_\nu^m h_k(t)=t^{m-1}\underbrace{(t^{-1} (t^{-1}\cdots(t^{-1}}_{m-1 \textrm{ times}}
w_k'(t)\underbrace{)'\cdots)')'}_{m-1\textrm{ times}}-\partial_\nu^m w_k(t).$$
Using the boundary condition $\partial_\nu^j w_k(t)=\partial_\nu^j h_k(t)$ for 
$0\leq j\leq m-1$, and observing that on the right-hand side the terms involving 
$\partial_\nu^m w_k(t)$ cancel, we can replace $w_k$ by $h_k$ and it suffices to prove
$$-\partial_\nu^m h_k(t)=t^{m-1}\underbrace{(t^{-1} (t^{-1}\cdots(t^{-1}}_{m-1 \textrm{ times}}
h_k'(t)\underbrace{)'\cdots)')'}_{m-1\textrm{ times}}-\partial_\nu^m h_k(t).$$
But $\Delta^m h_k=0$ and radial symmetry imply that $h(r)=\sum_{i=0}^{m-1}\alpha_ir^{2i}$; so
$$\underbrace{(t^{-1} (t^{-1}\cdots(t^{-1}}_{m-1 \textrm{ times}}
h_k'(t)\underbrace{)'\cdots)')'}_{m-1\textrm{ times}}=0, $$
and \eqref{ratta} follows.

\medskip

\noindent\emph{Step 2.} Inserting now \eqref{denu} into \eqref{ratta}, for any given $b<2$
we infer
\begin{equation}\label{stima3}
\begin{split}
&(-1)^{m-1}t^{m-1}\underbrace{\Big(t^{-1} \Big(t^{-1}\cdots\Big(t^{-1}}_{m-1 \textrm{ times}}
\Big(w_k'(t)+\frac{b}{t}\Big)\underbrace{\Big)'\cdots\Big)'\Big)'}_{m-1\textrm{ times}}\\
&=(-1)^{m-1}\partial_\nu^m g_k(t)+\frac{2^{m-1}(m-1)! b}{t^m}\leq 0 \ \textrm{ for } Lr_k\leq t\leq R, 
\end{split}
\end{equation}
provided that we fix $L = L(b)$ sufficiently large and then also choose $k$ large enough. 
We now prove by induction over $1 \le j \le m$ that
\begin{equation}\label{stima4}
\varphi_{j,k}(t):=(-1)^{m-j}t^{-1}\underbrace{\Big(t^{-1} \Big(t^{-1}\cdots\Big(t^{-1}}_{m-j 
\textrm{ times}}\Big(w_k'(t)+\frac{b}{t}\Big)\underbrace{\Big)'\cdots\Big)'\Big)'}_{m-j
\textrm{ times}}\leq P_j(t),
\end{equation}
for $Lr_k\leq t\leq R$, where $P_j(t)\ge 0$ is a polynomial in $t$ independent of $k$.
The case $j=1$ follows at once from \eqref{stima3} with $P_1\equiv 0$. Using the Dirichlet 
boundary condition (which implies $\partial_\nu^j w_k(R)=0$ for $1\leq j\leq m-1$) we get 
$\varphi_{j,k}(R)\le C_j$ for some constant $C_j \ge 0$, $2\leq j\leq m$. Observing that 
$\varphi_{j,k}'(t)=-t\varphi_{j-1,k}(t)$ for $2\leq j\leq m$, we then obtain 
\begin{eqnarray*}
\varphi_{j,k}(t)&=&\varphi_{j,k}(R)+\int_t^R r \varphi_{j-1,k}(r) dr\\
&\leq& C_{j}+\int_t^R r P_{j-1}(r)dr=:P_{j}(t),
\end{eqnarray*}
that is, \eqref{stima4}.
For $j=m$ we get
$$w_k'(t)\leq -\frac{b}{t}+tP_m(t), \ Lr_k\leq t\leq R.$$
Integrating once more, recalling that $L$ depends on $b$, and that $w_k(Lr_k)\to \eta_0(L)$
as $k \to \infty$, for sufficiently large $k$ we find
$$w_k(t)\leq w_k(Lr_k)-b\log\Big(\frac{t}{Lr_k}\Big)+C\leq b\log\Big(\frac{r_k}{t}\Big)+C$$
for $Lr_k\le t\le R$. For $0<t<Lr_k$ \eqref{stima1} already follows from Theorem \ref{trm1}.

In order to prove \eqref{stima0.5}, observe that \eqref{parti5} implies
$$(-\partial_\nu)^m g_k(t)\geq 0\ \textrm{ for } 0< t \leq R,\ k \in \N,$$
and
\eqref{ratta} yields
\begin{equation}\label{stima5}
(-1)^{m-1}t^{m-1}\underbrace{(t^{-1} (t^{-1}\cdots(t^{-1}}_{m-1 \textrm{ times}}
w_k'(t)\underbrace{)'\cdots)')'}_{m-1\textrm{ times}}\leq 0.
\end{equation}
In analogy with \eqref{stima4}, for $1\leq j\leq m$ we can show by induction that 
$$\psi_{j,k}(t):=(-1)^{m-j}t^{-1}\underbrace{(t^{-1} (t^{-1}\cdots(t^{-1}}_{m-j \textrm{ times}}
w_k'(t)\underbrace{)'\cdots)')'}_{m-j\textrm{ times}}\leq 0\ \textrm{ for all } 0<t\leq R.
$$
Indeed $\psi_{1,k}(t)\leq 0$ by \eqref{stima5}, while for $2\leq j\leq m$, we have 
$\psi_{j,k}(R)=0$ thanks to the boundary condition. Hence
$$\psi_{j,k}(t)=\int_t^R r \psi_{j-1}(r) dr \leq 0\ \textrm{ for all } 0<t\leq R,$$ 
and the case $j=m$ implies \eqref{stima0.5}.

\medskip

\noindent\emph{Step 3.} In order to prove \eqref{stima1c}, assume by contradiction that 
$$\liminf_{k\to\infty}\frac{T_k}{r_k}=L\in [0,\infty[.$$ 
Then from Theorem \ref{trm1} for a suitable subsequence  on the one hand we have 
$$u_k(0)(u_k(T_k)-u_k(0))+ \log 2
=\eta_k\bigg(\frac{T_k}{r_k}\bigg)\to \log\bigg(\frac{2}{1+L^2}\bigg)
\textrm{ as } k\to\infty.$$
But on the other hand, since $u_k(0)\to\infty$ we also have that 
$$u_k(0)(u_k(T_k)-u_k(0))=u_k^2(0)(\ve-1)\to\; -\infty$$
as $k\to\infty$, a contradiction. 

It thus remains to prove \eqref{stima1b}. Using \eqref{stima1} and observing that
$$(\ve -1)u_k^2(0)\leq w_k(r)\leq 0\ \textrm{ for }0\leq r\leq T_k,$$
from \eqref{rik} for $k \ge k_0$ we get
\begin{eqnarray*}
f_k(r)&\leq& \lambda_k u_k^2(0)e^{m u_k^2(0)}e^{2m\big(1+\frac{w_k(r)}{2u_k^2(0)}\big)w_k(r)}\\
&\leq& \lambda_k r_k^{2m}u_k^2(0)e^{mu_k^2(0)}r_k^{-2m}e^{m(\ve+1)w_k(r)}
\leq Cr_k^{-2m}\Big(\frac{r_k}{r}\Big)^{m(\ve+1)b}.
\end{eqnarray*}
Choosing now $b<2$ such that $m(\ve+1)b=2m+\ve$, and integrating over $B_{T_k}$, we find
\begin{eqnarray*}
\Lambda_1&\leq&\lim_{k\to\infty}\Lambda_k(T_k)\leq\lim_{k\to\infty}\sigma_k(T_k)\\
&=& \Lambda_1+\lim_{L\to\infty}\lim_{k\to\infty}\int_{B_{T_k}\backslash B_{Lr_k}}f_k dx\\
&\leq&\Lambda_1+C\lim_{L\to\infty}\lim_{k\to\infty}\frac{1}{r_k^{2m}}
\int_{B_{T_k}\backslash B_{Lr_k}}\Big(\frac{r_k}{r}\Big)^{2m+\ve}dx\\
&\leq& \Lambda_1+\frac{C}{\ve}\lim_{L\to\infty}\lim_{k\to\infty}
\Big(\frac{r_k}{Lr_k}\Big)^\ve= \Lambda_1,
\end{eqnarray*}
hence \eqref{stima1b}.
\end{proof}

According to Lemma \ref{lemma2} we can now choose a sequence $\ve_k\to 0$ as 
$k\to\infty$ and corresponding numbers $s_k=T_k(\ve_k)$ such that $u_k(s_k)\to\infty$ 
as $k\to\infty$ and
\begin{equation}\label{eq23}
\lim_{k\to\infty}\frac{r_k}{s_k}=0,\quad \lim_{k\to\infty}\Lambda_k(s_k)=\Lambda_1,
\quad \lim_{L\to\infty} \lim_{k\to\infty}N_k(Lr_k,s_k)=0.
\end{equation}
Observing that Theorem \ref{trm1} implies $\lim_{k\to\infty}\frac{u_k(Lr_k)}{u_k(0)}=1$ 
for every $L\geq 0$, we get
\begin{equation}\label{eq23c}
\lim_{k\to\infty}\frac{u_k(s_k)}{u_k(Lr_k)}=\lim_{k\to\infty}\frac{u_k(s_k)}{u_k(0)}=0,
\quad \text{for all }L>0.
\end{equation}
We also claim that
\begin{equation}\label{Lsk}
\lim_{L\to\infty}\lim_{k\to\infty}N_k(s_k,Ls_k)=0.
\end{equation}
To see this, remember that for $0<s<t<R$
$$N_k(s,t)=\int_{B_t\backslash B_s}e_k dx
  =\omega_{2m-1}\int_s^t \lambda_k r^{2m-1}u_k^2 e^{mu_k^2}dr.$$
Now set
$$P_k(t):=t\frac{\partial}{\partial t}N_k(s,t)
  =t\int_{\partial B_t}e_k d\sigma=\omega_{2m-1}\lambda_k t^{2m}u_k^2(t)e^{mu_k^2(t)}.$$
Using the monotonicity of $u_k$ that we proved in Lemma \ref{lemma2} we immediately obtain
the estimate
\begin{equation}\label{pkt}
P_k(t) =C\omega_{2m-1}\lambda_k u_k^2(t)e^{mu_k^2(t)}\int_{t/2}^t r^{2m-1}dr
\leq CN_k(t/2,t)\leq C P_k(t/2)
\end{equation}
analogous to (26) in \cite{str}; hence we also conclude that
\begin{equation}\label{pkt2}
N_k(t,2t)\leq CN_k(t/2,t) \quad\text{for } t\in [0,R/2].
\end{equation}
Now \eqref{eq23} and \eqref{pkt2} imply that for any $M\in\mathbb{N}$
\begin{eqnarray*}
\lim_{k\to\infty}N_k(2^{M-1}s_k,2^Ms_k)&\leq& C\lim_{k\to\infty}N_k(2^{M-2}s_k,s^{M-1}s_k)\\
&\leq&\cdots\leq C_M\lim_{k\to\infty} N_k(s_k/2,s_k)=0.
\end{eqnarray*}
Therefore if $2^M\geq L$ we have
$$\lim_{k\to\infty}N_k(s_k,Ls_k)\leq \sum_{j=1}^M N_k(2^{j-1}s_k,2^js_k)=0,$$
as claimed.

Setting $r_k^{(1)}:=r_k$, $s_k^{(1)}:=s_k$ and taking into account \eqref{eq23} - \eqref{Lsk} and Theorem \ref{trm1} we see that the property $(H_1)$ is satisfied.

\subsection{The inductive step}

We now assume that $(H_\ell)$ holds for some integer $\ell\geq 1$ and fix numbers
$$s_k^{(0)}=0 < r_k^{(1)}< s_k^{(1)}<\ldots< r_k^{(\ell)}<s_k^{(\ell)},\ k\in\mathbb{N}$$
such that $(H_{\ell,1})$, $(H_{\ell,2})$, $(H_{\ell,3})$ and $(H_{\ell,4})$ hold true.
To complete the proof of Theorem \ref{trm2} it suffices to show that either $(H_{\ell+1})$ 
or \eqref{Nk0} holds. The proof requires the following analogue of (29) in \cite{str}. 

\begin{lemma}\label{key}
There is a constant $C_0=C_0(\Lambda)$ such that for $t_k> s_k^{(\ell)}$ there holds
\begin{equation}\label{poho}
N_k(s_k^{(\ell)},t_k)\leq \frac{P_k(t_k)}{m}+C_0N_k^2(s_k^{(\ell)},t_k)+o(1),
\end{equation}
with error $o(1)\to 0$ as $k\to\infty$
\end{lemma}

\begin{proof} For $s=s_k^{(\ell)}<t$ we integrate by parts to obtain
\begin{equation}\label{Nkstima}
\begin{split}
& N_k(s,t)=\omega_{2m-1}\int_s^t r^{2m-1}\lambda_ku_k^2e^{mu_k^2}dr\\
&\quad =\lambda_k\frac{\omega_{2m-1}}{2m}\big(r^{2m}u_k^2 e^{mu_k^2}\big)\big|_s^t
-\frac{\omega_{2m-1}}{2m}\int_s^t \lambda_k r^{2m}(2u_k+2mu_k^3)u_k' e^{mu_k^2}dr\\
&\quad \leq \frac{P_k(t)}{2m}-\omega_{2m-1}\int_s^t\lambda_k r^{2m}\Big(\frac{1}{m}+u_k^2\Big)
\frac{u_k}{u_k(0)}w_k' e^{mu_k^2}dr.
\end{split}
\end{equation}
Define $g_k(t)$ as in the beginning of the proof of Lemma \ref{lemma2}. 
Then \eqref{parti5} and \eqref{ratta} imply
$$\underbrace{(t^{-1} (t^{-1}\cdots (t^{-1}}_{m-1 \textrm{ times}}w_k'(t)
\underbrace{)'\cdots)')'}_{m-1\textrm{ times}}=(-1)^m\frac{A_k(t)}{\omega_{2m-1}t^{4m-3}},$$
where $A_k$ is as in \eqref{parti6}.
Integrating this relation $m-1$ times from $t$ to $R$, and using the Dirichlet boundary 
condition $\de_\nu^j w_k(R)=0$ for $1\leq j\leq m-1$ we get
$$\frac{w_k'(t)}{t}=-\int_t^Rt_1\int_{t_1}^Rt_2\cdots
\int_{t_{m-2}}^R\frac{A_k(t_{m-1})}{\omega_{2m-1}t_{m-1}^{4m-3}}dt_{m-1}\cdots dt_1;$$
hence
$$-t w_k'(t)\frac{u_k(t)}{u_k(0)}=t^2\frac{u_k(t)}{u_k(0)}\int_t^R t_1\int_{t_1}^Rt_2\cdots
\int_{t_{m-2}}^R\frac{A_k(t_{m-1})}{\omega_{2m-1}t_{m-1}^{4m-3}}dt_{m-1}\cdots dt_1=:I.$$
More explicitly, 
\begin{equation*}
 \begin{split}
   I&=t^2\int_t^Rt_1\int_{t_1}^Rt_2\cdots\int_{t_{m-2}}^R\frac{1}{\omega_{2m-1}t_{m-1}^{4m-3}}\\
    &\times\int_0^{t_{m-1}}\rho_1\cdots\int_0^{\rho_{m-2}}\rho_{m-1}\tau_k(\rho_{m-1},t)
   d\rho_{m-1}\cdots d\rho_1\;dt_{m-1}\cdots dt_1,
 \end{split}
\end{equation*}
where
\begin{equation*}
 \tau_k(\rho,t) = \frac{u_k(t)\sigma_k(\rho)}{u_k(0)}
 = \int_{B_{\rho}}\lambda_ku_k(t)u_k e^{mu_k^2}dx.
\end{equation*}
We now show that $I$ can be bounded in terms of $N_k(s,t)$ up to a small error. 
From this the desired inequality \eqref{poho} will be immediate. Split
$$I=:II+III,$$
where $II$ corresponds to $\rho_{m-1}\leq t$. Since $u_k'\leq 0$, for $\rho\leq t$ we have 
\begin{equation}\label{3.1}
 \begin{split}
  \tau_k(\rho,t)& =\int_{B_{\rho}}\lambda_ku_k(t)u_k e^{mu_k^2}dx
  \leq \int_{B_{\rho}}\lambda_ku_k(\rho)u_k e^{mu_k^2}dx\\
  & \leq \int_{B_s}\lambda_k u_k(s)u_ke^{mu_k^2}dx+ N_k(s,\rho)
  \leq N_k(s,t)+o(1)
 \end{split}
\end{equation}
with error $o(1)\to 0$ as $k\to\infty$. Here we used that for arbitrary $L>1$ we can bound
\begin{equation*}
 \int_{B_s}\lambda_k u_k(s)u_ke^{mu_k^2}dx \leq
 N_k(Lr_k^{(\ell)},s)+\frac{u_k(s)}{u_k(Lr_k^{(\ell)})}\Lambda_k(Lr_k^{(\ell)}),
\end{equation*}
and by $(H_{\ell,2})$, $(H_{\ell,4})$ the latter tends to $0$, if first $k\to\infty$ and then $L\to\infty$.
Since 
\begin{equation*}
t^2\int_t^\infty t_1\cdots\int_{t_{m-2}}^ \infty\frac{1}{\omega_{2m-1}t_{m-1}^{4m-3}}
\int_0^{t_{m-1}}\rho_1\cdots\int_0^{\rho_{m-2}}\rho_{m-1}d\rho_{m-1}\cdots dt_1 \le C
\end{equation*}
uniformly in $t$, we conclude that
\begin{equation*}
  II\leq C N_k(s,t)+o(1).
\end{equation*}
In order to obtain a similar bound for $III$, for $t \le \rho$ we estimate  
\begin{equation*}
 \tau_k(\rho,t) = \frac{u_k(t)\sigma_k(\rho)}{u_k(0)}
 = \frac{u_k(t)}{u_k(\rho)+1}\int_{B_{\rho}}\lambda_k(u_k(\rho)+1)u_k e^{mu_k^2}dx\ .
\end{equation*}
Recalling \eqref{3.1}, we have
\begin{equation*}
 \int_{B_{\rho}}\lambda_k(u_k(\rho)+1)u_k e^{mu_k^2}dx 
 \le \tau_k(\rho,\rho) + o(1) \le N_k(s,\rho)+o(1).
\end{equation*}
Also note that by H\"older's inequality we can estimate
\begin{equation*}
  |u_k(t)-u_k(\rho)| \le \int_t^{\rho}|u_k'(r)|dr 
  \le \|\nabla u_k\|_{L^{2m}}
  \big(\log\frac{\rho}{t}\big)^\frac{2m-1}{2m}.
\end{equation*}
Thus, with a constant $C=C(\Lambda)$ for all $t \le \rho$ we obtain
\begin{equation}\label{3.2}
 \begin{split}
   \frac{tu_k(t)}{\rho (u_k(\rho)+1)} & =\frac{t}{\rho}\bigg(\frac{u_k(t)-u_k(\rho)}{u_k(\rho)+1}
   +\frac{u_k(\rho)}{u_k(\rho)+1}\bigg)\\
   & \leq \frac{t}{\rho}\bigg(C\big(\log\frac{\rho}{t}\big)^\frac{2m-1}{2m}+1\bigg)\le C
 \end{split}
\end{equation}
and with $C_1=C_1(\Lambda)$ we can bound 
\begin{equation*}
  \frac{t}{\rho}\tau_k(\rho,t)\le C_1 N_k(s,\rho)+o(1).
\end{equation*}
It follows that
\begin{equation*}
 \begin{split}
   III&=t^2\int_t^Rt_1\int_{t_1}^Rt_2\cdots\int_{t_{m-2}}^R\frac{1}{\omega_{2m-1}t_{m-1}^{4m-3}}\\
    &\quad \times\int_t^{t_{m-1}}\rho_1\cdots\int_t^{\rho_{m-2}}\rho_{m-1}\tau_k(\rho_{m-1},t)
     d\rho_{m-1}\cdots d\rho_1\;dt_{m-1}\cdots dt_1\\
    &=t\int_t^Rt_1\int_{t_1}^Rt_2\cdots\int_{t_{m-2}}^R\frac{1}{\omega_{2m-1}t_{m-1}^{4m-3}}\\
    &\quad \times\int_t^{t_{m-1}}\rho_1\cdots\int_t^{\rho_{m-2}}\rho^2_{m-1}\frac{t}{\rho_{m-1}}
     \tau_k(\rho_{m-1},t)d\rho_{m-1}\cdots d\rho_1\;dt_{m-1}\cdots dt_1\\
    &\le C_1t\int_t^Rt_1\int_{t_1}^Rt_2\cdots\int_{t_{m-2}}^R\frac{1}{\omega_{2m-1}t_{m-1}^{4m-3}}\\
    &\quad \times\int_t^{t_{m-1}}\rho_1\cdots\int_t^{\rho_{m-2}}\rho^2_{m-1}(N_k(s,\rho_{m-1})+o(1))
    d\rho_{m-1}\cdots dt_1.
 \end{split}
\end{equation*}
For any $L\geq 1$ we split the integral with respect to $t_1$ and use the obvious inequality
$N_k(s,\rho_{m-1})\leq 2\Lambda$ for large $k$ to estimate
\begin{equation*}
 \begin{split}
   III&\le C_1t\int_t^{Lt}t_1\int_{t_1}^Rt_2\int_{t_2}^Rt_3\cdots\int_{t_{m-2}}^R\frac{1}{\omega_{2m-1}t_{m-1}^{4m-3}}\\
    &\qquad \times\int_t^{t_{m-1}}\rho_1\cdots\int_t^{\rho_{m-2}}\rho^2_{m-1}(N_k(s,\rho_{m-1})+o(1))
    d\rho_{m-1}\cdots dt_1\\
    &\quad + 2C_1\Lambda t\int_{Lt}^Rt_1\int_{t_1}^Rt_2\int_{t_2}^Rt_3\cdots\int_{t_{m-2}}^R\frac{1}{\omega_{2m-1}t_{m-1}^{4m-3}}\\
    &\qquad \times\int_t^{t_{m-1}}\rho_1\cdots\int_t^{\rho_{m-2}}\rho^2_{m-1}d\rho_{m-1}\cdots dt_1+o(1)\\
 \end{split}
\end{equation*}
Observing the uniform bound
\begin{equation*}
Lt\int_{Lt}^\infty t_1\cdots\int_{t_{m-2}}^ \infty\frac{1}{\omega_{2m-1}t_{m-1}^{4m-3}}
\int_0^{t_{m-1}}\rho_1\cdots\int_0^{\rho_{m-2}}\rho^2_{m-1}d\rho_{m-1}\cdots dt_1 \le C, 
\end{equation*}
with a constant $C_2=C_2(\Lambda)$ we obtain 
\begin{equation*}
 \begin{split}
   III&\le C_1t\int_t^{Lt}t_1\int_{t_1}^Rt_2\int_{t_2}^Rt_3\cdots\int_{t_{m-2}}^R\frac{1}{\omega_{2m-1}t_{m-1}^{4m-3}}\\
    &\qquad \times\int_t^{t_{m-1}}\rho_1\cdots\int_t^{\rho_{m-2}}\rho^2_{m-1}(N_k(s,\rho_{m-1})+o(1))
    d\rho_{m-1}\cdots dt_1\\
    &\quad + \frac{C_2\Lambda}{L}+o(1)\\
 \end{split}
\end{equation*}
To proceed we successively split the integral also with respect to $t_2,\dots,t_{m-1}$ and 
use the uniform bounds
\begin{equation*}
 \begin{split}
  Lt\int_0^{Lt}t_1\cdots&\int_0^{Lt}t_{j-1}\int_{Lt}^\infty t_j\cdots\int_{t_{m-2}}^\infty\frac{1}{\omega_{2m-1}t_{m-1}^{4m-3}}\\
  &\times \int_0^{t_{m-1}}\rho_1\cdots\int_0^{\rho_{m-2}}\rho^2_{m-1}d\rho_{m-1}\cdots dt_j \le C
 \end{split}
\end{equation*}
for $2 \le j < m$ to estimate
\begin{equation*}
 \begin{split}
   III&\le C_1t\int_t^{Lt}t_1\int_{t_1}^{Lt}t_2\int_{t_2}^Rt_3\cdots\int_{t_{m-2}}^R\frac{1}{\omega_{2m-1}t_{m-1}^{4m-3}}\\
    &\qquad \times\int_t^{t_{m-1}}\rho_1\cdots\int_t^{\rho_{m-2}}\rho^2_{m-1}(N_k(s,\rho_{m-1})+o(1))
    d\rho_{m-1}\cdots dt_1\\
    & \quad + 2C_1\Lambda t\int_t^{Lt}t_1\int_{Lt}^Rt_2\int_{t_2}^Rt_3\cdots\int_{t_{m-2}}^R\frac{1}{\omega_{2m-1}t_{m-1}^{4m-3}}\\
    &\qquad \times\int_t^{t_{m-1}}\rho_1\cdots\int_t^{\rho_{m-2}}\rho^2_{m-1}d\rho_{m-1}\cdots dt_1+\frac{C_2\Lambda}{L}+o(1)\\
    & \leq \cdots \leq C_1 t\int_t^{Lt} t_1\int_{t_1}^{Lt}t_2\int_{t_2}^{Lt}t_3\cdots\int_{t_{m-2}}^{Lt}\frac{1}{\omega_{2m-1}t_{m-1}^{4m-3}}\\
   &\qquad \ \times \int_t^{t_{m-1}}\rho_1\cdots\int_t^{\rho_{m-2}}\rho_{m-1}^2
   \big(N_k(s,Lt)+o(1)\big)d\rho_{m-1}\cdots dt_1\\
   &\quad +2C_1\Lambda t\int_{t}^{Lt} t_1\int_{t_1}^{Lt}t_2\int_{t_2}^{Lt}t_3\cdots\int_{Lt}^R\frac{1}{\omega_{2m-1}t_{m-1}^{4m-3}}\\
   &\qquad \ \times \int_t^{t_{m-1}}\rho_1\cdots\int_t^{\rho_{m-2}}\rho_{m-1}^2d\rho_{m-1}
    \cdots dt_1+\frac{C_{m-1}\Lambda}{L}+o(1)\\
   &\leq C_mN_k(s,Lt)+\frac{C_m\Lambda}{L}+o(1),
 \end{split}
\end{equation*}
with constants $C_j=C_j(\Lambda)$, $2\le j \le m$.
Using \eqref{Lsk} in case $t \le 2s$ and \eqref{pkt} in case $t > 2s$ we get
$$N_k(s,Lt)\leq C(L)N_k(s,t)+o(1),$$
and with the constant $C_{m+1}=C_m\Lambda=C_{m+1}(\Lambda)$ there results 
$$-tw_k'(t)\frac{u_k(t)}{u_k(0)}\leq C(L,\Lambda)N_k(s,t)+\frac{C_{m+1}}{L}+o(1).$$
Inserting this into \eqref{Nkstima} we infer
\begin{eqnarray}
N_k(s,t)&\leq& \frac{P_k(t)}{2m}-\omega_{2m-1}\int_s^t\lambda_k r^{2m}
\Big(\frac{1}{m}+u_k^2\Big)\frac{u_k}{u_k(0)}w_k' e^{mu_k^2}dr\nonumber\\
&\leq&\frac{P_k(t)}{2m}+\Big(C(L,\Lambda)N_k(s,t)+\frac{C_{m+1}}{L}\Big) N_k(s,t)+o(1).\label{stima8}
\end{eqnarray}
Choosing $L=2C_{m+1}$ we finally get \eqref{poho} for an appropriate $C_0=C_0(\Lambda)$.
\end{proof}

\begin{lemma}\label{key2}
Let $C_0=C_0(\Lambda)$ be the constant appearing in \eqref{poho}. 
If for some $t_k\in] s_k^{(\ell)},R]$ there holds
\begin{equation}\label{0alpha}
0<\lim_{k\to \infty} N_k(s_k^{(\ell)},t_k)=:\alpha<\frac{1}{2C_0},
\end{equation}
then 
$$\lim_{k\to\infty}\frac{s_k^{(\ell)}}{t_k}=0,\ 
\liminf_{k\to\infty} P_k(t_k)\geq \frac{m\alpha}{2}, \text{ and } 
\lim_{L\to\infty}\lim_{k\to\infty} N_k(s_k^{(\ell)},t_k/L)=0.$$
\end{lemma}

\begin{proof}
Assume that for some $t_k\in ]s_k^{(\ell)},R]$ we have \eqref{0alpha}.
Since the same reasoning as in the proof of 
\eqref{Lsk} also yields that
$$\lim_{L\to\infty}\lim_{k\to\infty} N_k(s_k^{(\ell)},Ls_k^{(\ell)})=0,$$
necessarily $s_k^{(\ell)}/t_k\to 0$ as $k\to\infty$.
Moreover, \eqref{poho} yields
\begin{equation}\label{eq31}
 \begin{split}
   \liminf_{k\to\infty}\frac{P_k(t_k)}{m}& \geq\lim_{k\to\infty}\Big(N_k(s_k^{(\ell)},t_k)-
   C_0 N^2_k(s_k^{(\ell)},t_k)\Big)\\
   &\geq \lim_{k\to \infty}\frac{N_k(s_k^{(\ell)},t_k)}{2}=\frac{\alpha}{2},
 \end{split}
\end{equation}
as claimed. Now we show that
\begin{equation}\label{eq31b}
\lim_{L\to\infty}\lim_{k\to\infty}N_k(s_k^{(\ell)},t_k/L)=0.
\end{equation}
Indeed, if we assume
$$\lim_{L\to\infty}\limsup_{k\to\infty}N_k(s_k^{(\ell)},t_k/L)=\beta>0,$$
we have
$$\frac{\beta}{2}\leq N_k(s_k^{(\ell)},t_k/L)\leq N_k(s_k^{(\ell)},t_k)<\frac{1}{2C_0}$$
for any $L\geq 1$ and sufficiently large $k$. Therefore
we can apply \eqref{eq31} with $t_k/L$ instead of $t_k$ for any $L\geq 1$ to get
$$\lim_{k\to\infty}P_k(t_k/L)\geq \frac{m\beta}{2}.$$
Then \eqref{pkt} yields
$$C\lim_{k\to\infty}N_k(t_k/(2L),t_k/L)\geq \lim_{k\to\infty}P_k(t_k/L)\geq\frac{m\beta}{2}.$$
Choosing $L=2^j$ and summing over $j$ for $0\leq j\leq M-1$, we get
$$C\lim_{k\to\infty}\Lambda_k(t_k)\geq C\lim_{k\to\infty}N_k(2^{-M}t_k,t_k)
\geq \frac{mM\beta}{2}\to \infty\quad \text{as } M\to\infty,$$
which contradicts \eqref{eq1}. Therefore \eqref{eq31b} is proven.
\end{proof}

Suppose now that for some $t_k\geq s_k^{(\ell)}$ there holds
$$\limsup_{k\to\infty}N_k(s_k^{(\ell)},t_k)>0.$$
We then want to show that $(H_{\ell+1})$ holds. We can choose numbers 
$r_k^{\ell+1}\in]s_k^{(\ell)},t_k[$ such that for a subsequence there holds
\begin{equation}\label{eq32}
0<\lim_{k\to\infty}N_k(s_k^{(\ell)},r_k^{(\ell+1)})<\frac{1}{2C_0},
\end{equation}
where $C_0$ is as in Lemma \ref{key2}. Observe that Lemma \ref{key2} then implies
\begin{equation}\label{eq33}
\lim_{k\to \infty}s_k^{(\ell)}/r_k^{(\ell+1)}
=\lim_{L\to\infty}\lim_{k\to\infty}N_k(s_k^{(\ell)},r_k^{(\ell+1)}/L)=0,
\end{equation}
and
\begin{equation}\label{eq34}
\lim_{k\to\infty}P_k(r_k^{(\ell+1)})>0.
\end{equation}

\begin{prop}\label{blow} We have
$$\eta_k^{(\ell+1)}(x):=u_k(r_k^{(\ell+1)})\big(u_k(r_k^{(\ell+1)}x)-u_k(r_k^{(\ell+1)})\big)
\to \eta^{(\ell+1)}$$
in $C^{2m-1}_{\loc}(\R{2m}\backslash\{0\})$. Moreover, for a suitable constant  
$c^{(\ell+1)}$ the function  $\eta_0^{(\ell+1)}:=\eta^{(\ell+1)}+ c^{(\ell+1)}$ satisfies
$$(-\Delta)^m\eta_0^{(\ell+1)}=(2m-1)!e^{2m \eta_0^{(\ell+1)}},
\quad \int_{\R{2m}}(2m-1)! e^{2m \eta_0^{(\ell+1)}}dx=\Lambda_1.$$
\end{prop}

The above proposition, which will be proven in the following section,  implies that
$$\lim_{L\to\infty}\lim_{k\to\infty}N_k(r_k^{(\ell+1)}/L,Lr_k^{(\ell+1)})=\Lambda_1;$$
hence \eqref{eq33} yields
\begin{equation}\label{3.3}
  \begin{split}
    & \lim_{L\to\infty}\lim_{k\to\infty} N_k(s_k^{(\ell)},Lr_k^{(\ell+1)})\\
    & \quad =\lim_{L\to\infty}\lim_{k\to\infty}N_k(s_k^{(\ell)},r_k^{(\ell+1)}/L)
    +\lim_{L\to\infty}\lim_{k\to\infty}N_k(r_k^{(\ell+1)}/L,Lr_k^{(\ell+1)})\\
    & \quad =0+\Lambda_1=\Lambda_1.
  \end{split}
\end{equation}
Then the inductive hypothesis $(H_{\ell,3})$ gives
\begin{eqnarray*}
\lim_{L\to\infty}\lim_{k\to\infty}\Lambda_k(Lr_k^{(\ell+1)})
&=&\lim_{L\to\infty}\lim_{k\to\infty}\Big(\Lambda_k(s_k^{(\ell)})
+N_k(s_k^{(\ell)},Lr_k^{(\ell+1)})\Big)\\
&=&(\ell+1)\Lambda_1.
\end{eqnarray*}
Now set $w_k^{(\ell+1)}(x)=u_k(r_k^{(\ell+1)})(u_k(x)-u_k(r_k^{(\ell+1)}))$ so that
$$(-\Delta)^m w_k^{(\ell+1)}=\lambda_ku_k(r_k^{(\ell+1)})u_ke^{mu_k^2}=:f_k^{(\ell+1)}.$$
Similar to Lemma \ref{lemma2} and with the same proof (except that instead of 
Theorem \ref{trm1} one needs to use Proposition \ref{blow}) we have

\begin{lemma}\label{lemma2bis} For any $0 < \varepsilon < 1$, 
letting $T_k^{(\ell+1)}(\varepsilon)>0$ be such that 
$u_k(T_k^{(\ell+1)})=\varepsilon u_k(r_k^{(\ell+1)})$, we have
\begin{equation}\label{eq39}
\lim_{k\to\infty}N_k(s_k^{(\ell)},T_k^{(\ell+1)})=\Lambda_1.
\end{equation}
Moreover $r_k^{(\ell+1)}/T_k^{(\ell+1)} \to 0$ as $k\to\infty$.
\end{lemma}

According to Lemma \ref{lemma2bis} and \eqref{3.3} we can choose numbers $\varepsilon_k\to 0$ 
and a subsequence so that for
$s_k^{(\ell+1)}:=T_k^{(\ell+1)}(\varepsilon_k)$ we have
$u_k(s_k^{(\ell+1)}) \to \infty$ as $k \to \infty$ and
$$\lim_{k\to\infty}\frac{r_k^{(\ell+1)}}{s_k^{(\ell+1)}}=0,$$
while
$$\lim_{k\to\infty}\Lambda_k(s_k^{(\ell+1)})=(\ell+1)\Lambda_1, 
\quad \lim_{L\to\infty}\lim_{k\to\infty} N_k(Lr_k^{(\ell+1)},s_k^{(\ell+1)})=0.$$
Again reasoning as in the proof of \eqref{Lsk} we also infer
$$\lim_{k\to\infty}N_k(s_k^{(\ell+1)},Ls_k^{(\ell+1)})=0\quad 
\text{for every }L\geq 1.$$
Finally, observe that the definition of $s_k^{(\ell+1)}$ implies that 
$$\lim_{k\to\infty}\frac{u_k(s_k^{(\ell+1)})}{u_k(Lr_k^{(\ell+1)})}=0\quad 
\text{for every }L\geq 0.$$
Together with \eqref{eq33} this completes the proof of $(H_{\ell+1})$, 
and hence of Theorem \ref{trm2} in the radially symmetric case.

\subsection{Proof of Proposition \ref{blow}}
As preparation for the proof of Proposition \ref{blow} we need the following two lemmas.

\begin{lemma}\label{lemma2.1} For $r_k^{(\ell+1)}$ as above, we have
$$v_k(x):=u_k\big(r_k^{(\ell+1)}x\big)-u_k\big(r_k^{(\ell+1)}\big)\to 0\quad 
\textrm{in }C^{2m-1}_{\loc}(\R{2m}\backslash\{0\}).$$
\end{lemma}

\begin{proof} We write $r_k=r_k^{(\ell+1)}$. Moreover, we consider only the case $m>1$,
the case $m=1$ being considerably easier. As in the proof of Lemma 3.2 in \cite{str}
we have
$$(-\Delta)^m v_k(x)=\lambda_k r_k^{2m}u_k(r_k x)e^{m u_k^2(r_kx)}=:g_k(x)\geq 0,$$
with $g_k\to 0$ in $L^\infty_{\loc}(\R{2m}\backslash\{0\})$. 
By scaling and Sobolev's embedding we also have
\begin{equation}\label{eq1.1}
  \begin{split}
   \|\nabla^2 v_k\|_{L^m(B_{R/r_k})}&=\|\nabla^2 u_k\|_{L^m(B_R)}\leq C,\\
  \|\nabla^m v_k\|_{L^2(B_{R/r_k})}&=\|\nabla^m u_k\|_{L^2(B_R)}\leq C.
  \end{split}
\end{equation}
Set $w_k:=\Delta v_k$. Then a subsequence $w_k\to w$ weakly in $H^{m-2}_{\loc} (\R{2m})$ 
and in $C^{2m-3,\alpha}_{\loc}(\R{2m}\setminus \{0\})$ for some function $w\in L^m(\R{2m})$ with
$\nabla^{m-2}w\in L^2(\R{2m})$. Clearly $\Delta^{m-1} w=0$ in $\R{2m}\backslash \{0\}$. In fact,
since the point $x=0$ has vanishing $H^m$-capacity, as in \cite{str} we have $\Delta^{m-1} w=0$ 
in $\R{2m}$. Recalling that $w\in L^m(\R{2m})$ we conclude that $w\equiv 0$; see 
Lemmas \ref{trmliou} and \ref{lemcap} in the appendix.

Recalling that $(\Delta v_k)$ is bounded in $L^m(\R{2m})$ and noting the condition $v_k(1)=0$, 
from standard elliptic estimates we infer that $(v_k)$ is bounded in $W^{2,m}(B_1)$.
Hence a subsequence $v_k\to v$ weakly in 
$W^{2,m}(B_1)$ and in $C^{2m-1,\alpha}$ away  from $x=0$. 
We then have $\Delta v=0$ and $v(1)=0$, therefore $v\equiv 0$ on $B_1$.

By elliptic estimates, from \eqref{eq1.1} and the condition $v_k(1)=0$ we also infer that $(v_k)$ 
is bounded in $W^{2,m}_{\loc}(\R{2m})$. Therefore, we also have that $v_k\to v$ weakly 
in $W^{2,m}_{\loc}(\R{2m})$ and in $C^{2m-1,\alpha}_{\loc}(\R{2m}\setminus \{0\})$, 
with $\Delta v=0$. 
By unique continuation it follows that $v\equiv 0$. This completes the proof.
\end{proof}

\begin{lemma}\label{lemma3} For any $L>0$ there exists $k_0=k_0(L)$ such that
for all $k\geq k_0$ and any $1\leq j\leq 2m-1$ there holds
$$u_k(r_k^{(\ell+1)})\int_{B_{Lr_k^{(\ell+1)}}\setminus B_{r_k^{(\ell+1)}/L}}
|\nabla^j u_k|dx\leq C(Lr_k^{(\ell+1)})^{2m-j}.$$
\end{lemma}

\begin{proof} The proof is identical to the proof of Lemma 6 in \cite{mar2}, using 
Lemma \ref{lemma2.1} above instead of Lemma 3 in \cite{mar2}.
\end{proof}

\medskip

\noindent\emph{Proof of Proposition \ref{blow}.} For simplicity of notation, we now drop the 
index ${\ell}+1$.

\medskip

\noindent\emph{Step 1.} We claim that $\eta_k\to \eta$ in 
$C^{2m-1,\alpha}_{\loc}(\R{2m}\backslash \{0\})$ for some smooth function $\eta$. 
For any $L>1$ let $\Omega_L:=B_L(0)\backslash B_{1/L}(0)$. Recall that by 
Lemma \ref{lemma2.1} we have
$\overline u_k(x):=\frac{u_k(r_k x)}{u_k(r_k)}\to 1$ uniformly on $\Omega_L$ 
as $k\to \infty$. Thus by \eqref{eq9} with error $o(1) \to 0$ as $k\to \infty$ we have
\begin{equation}\label{2.1}
  \begin{split}
      0 & \le (-\Delta)^m \eta_k(x) =\lambda_k r_k^{2m}u_k^2(r_k)
      \overline u_k(x)e^{m u_k^2(r_k x)}\\
      &  \le (L^{2m}+o(1))\lambda_k (r_k|x|)^{2m}u_k^2(r_kx)
      e^{m u_k^2(r_k x)} \le CL^{2m}+o(1).
  \end{split}
\end{equation}
Split $\eta_k=h_k+l_k$ on $\Omega_{2L}$, where 
$$\Delta^m h_k=0\ \textrm{ on }\Omega_{2L}, \ \text{ and }
l_k=\Delta l_k=\ldots=\Delta^{m-1}l_k=0\ \textrm{ on } \partial \Omega_{2L}.$$
Since $\|\Delta^m\eta_k\|_{L^\infty(\Omega_{2L})}\leq C=C(L)$, by elliptic estimates we get
that $l_k\to l$ in $C^{2m-1,\alpha}(\Omega_{2L})$. Together with Lemma \ref{lemma3} this
implies
$$\|\nabla h_k\|_{L^1(\Omega_{2L})}
\leq \|\nabla l_k\|_{L^1(\Omega_{2L})}+\|\nabla \eta_k\|_{L^1(\Omega_{2L})}\leq C.$$
Moreover, since $\eta_k=0$ on $\partial B_1(0)$, we have
\begin{equation}\label{eq4.1}
|h_k|=|l_k|\leq C \quad\textrm{ on } \partial B_1(0).
\end{equation}
Then, from a Poincar\'e-type inequality, we easily get $\|h_k\|_{L^1(\Omega_{2L})}\leq C$.
By virtue of Proposition \ref{c2m}, we infer that
$$\|h_k\|_{C^j(\Omega_L)}\leq C_j\quad \textrm{for every } j\in \mathbb{N}.$$
Hence a subsequence $h_k\to h$ smoothly on $\Omega_L$, and
$$\eta_k\to \eta:=h+l \quad \textrm{in } C^{2m-1,\alpha}(\Omega_L),$$
proving our claim.

\medskip

\noindent \emph{Step 2.} With $\overline u_k(x):=\frac{u_k(r_k x)}{u_k(r_k)}$ as above, from 
\eqref{2.1} we get
\begin{equation}\label{2.2}
  \begin{split}
(-\Delta)^m \eta_k&=\lambda_k r_k^{2m}u_k^2(r_k)e^{m u_k^2(r_k)}
\overline u_k(x)e^{m(u_k^2(r_k \,\cdot\,)-u_k^2(r_k))}\\
&=\mu_k \overline u_k e^{m(\overline u_k +1)\eta_k}, 
  \end{split}
\end{equation}
where by \eqref{eq34} we may assume
$$\mu_k:=\lambda_k r_k^{2m} u_k^2(r_k)e^{m u_k^2(r_k)}=\omega_{2m-1}^{-1}P_k(r_k)\to \mu_0>0.$$
Since $\overline u_k\to 1$ locally uniformly on 
$\R{2m}\backslash \{0\}$ we may pass to the limit $k\to \infty$ in \eqref{2.2}
to see that $\eta$ solves the equation
\begin{equation}\label{eq5.1}
(-\Delta)^m \eta=\mu_0 e^{2m\eta} \quad \textrm{on }\R{2m}\backslash\{0\}
\end{equation}
in the distribution sense. In fact, we now show that \eqref{eq5.1} 
holds on all of $\R{2m}$. 
Note that by Step 1 for any $L>1$ we have
\begin{eqnarray*}
\int_{\Omega_L} e^{2m\eta}dx &=&
\lim_{k\to\infty}\int_{\Omega_L}  \overline u_k^2e^{m(\overline u_k +1)\eta_k}dx\\
&=&\lim_{k\to\infty} \int_{\Omega_L}\mu_k^{-1}\overline u_k (-\Delta)^m\eta_kdx\\
&\leq & \mu_0^{-1}\liminf_{k\to\infty}\int_{B_{Lr_k}} u_k(-\Delta)^m u_k dx\leq \mu_0^{-1}\Lambda. 
\end{eqnarray*}
As $L\to \infty$, by Fatou's lemma, we get $e^{2m\eta}\in L^1(\R{2m})$. Moreover 
$\eta\geq 0$ on $B_1$, hence $\eta\in L^p(B_1)$ for every $p\in [1,\infty[$.
Also note that $(-\Delta)^m \eta_k\ge 0$ and that from \eqref{3.1} we can bound
\begin{equation}\label{3.4}
  \begin{split} 
   & \limsup_{k\to \infty} \int_{B_{1/L}(0)} (-\Delta)^m \eta_k\; dx \\ 
   & \quad = \limsup_{k\to \infty} \int_{B_{r_k/L}(0)} 
     \lambda_k u_k(r_k) u_k e^{mu_k^2} dx 
   \le \limsup_{k\to \infty} N_k(s_k^{(\ell)},r_k/L) \rightarrow 0
  \end{split} 
\end{equation}
as $L \to \infty$. Since by Lemma \ref{lemma2} we have $\overline u_k \ge 1$, 
$\eta_k \ge 0$ on $B_1$, from \eqref{2.2} and \eqref{3.4} we also find that 
\begin{equation}\label{3.4a}
 \limsup_{k\to \infty} \int_{B_{1/L}(0)} \eta_k\; dx  \rightarrow 0 \text{ as } L \to \infty\; . 
\end{equation}
By \eqref{eq5.1} and \eqref{3.4} for any test function 
$\varphi \in C^{\infty}_0(\R{2m})$ we now obtain 
\begin{equation}\label{3.5}
  \begin{split} 
   \int_{\R{2m}}& \big((-\Delta)^m \eta-\mu_0e^{2m\eta}\big)\varphi\; dx
   = \lim_{L\to \infty}\int_{\R{2m}}(-\Delta)^m \eta\varphi\tau_L\; dx \\
   & = \lim_{L\to \infty} \liminf_{k\to \infty}\int_{\R{2m}}
   \big((-\Delta)^m \eta-(-\Delta)^m \eta_k\big)\varphi\tau_L\; dx, 
  \end{split} 
\end{equation}
where for $L \in \N$ we let $\tau_L(x)=\tau(Lx)$ with a fixed cut-off function
$\tau\in C^{\infty}_0(B_2)$ such that $0\le \tau \le 1$ 
and $\tau\equiv 1$ in $B_1$. But by Step 1 for any $L\ge 1$ we have 
\begin{equation*}
  \begin{split} 
   \liminf_{k\to \infty}& \int_{\R{2m}}
   \big((-\Delta)^m \eta-(-\Delta)^m \eta_k\big)\varphi\tau_L\; dx\\
   & =  \liminf_{k\to \infty}\int_{\R{2m}}
   (\eta-\eta_k)\big((-\Delta)^m \varphi\big)\tau_L\; dx,
  \end{split} 
\end{equation*}
and since $\eta\in L^1(B_1)$ and on account of \eqref{3.4a} the latter converges to $0$ as $L \to \infty$ for any fixed $\varphi \in C^{\infty}_0(\R{2m})$.
From \eqref{3.5} we thus see that $\eta$ solves \eqref{eq5.1} in the distribution sense on 
$\R{2m}$. By elliptic estimates, $\eta$ is smooth on all of $\R{2m}$; 
see for instance \cite{mar1}, Corollary 8. The function 
$\eta_0:=\eta+\frac{1}{2m}\log\frac{\mu_0}{(2m-1)!}$ then satisfies 
\begin{equation}\label{eq5.2}
(-\Delta)^m \eta_0=(2m-1)! e^{2m\eta_0}\quad \textrm{in }\R{2m},
\quad \int_{\R{2m}}e^{2m\eta_0}dx<\infty.
\end{equation}
Solutions to \eqref{eq5.2} have been classified in \cite{mar1}, where it was shown that either

\medskip

\noindent
(i) $\eta_0(x)=\log\frac{2\sigma}{1+\sigma|x-x_0|^2}$ for some $\sigma>0$, $x_0\in\R{2m}$, or

\medskip

\noindent
(ii) $m>1$ and there exist $1\leq j\leq m-1$ and $a\neq 0$ such that
\begin{equation*}
\lim_{|x|\to \infty}\Delta^j \eta_0 (x)=a,
\end{equation*}
and hence for sufficiently large $L$, with error $o(1) \to 0$ as $k \to \infty$,  
\begin{equation}\label{eqa}
\begin{split}
   (Lr_k)^{2j-2m} & u_k(r_k)\int_{B_{Lr_k}\setminus B_{r_k/L}}|\nabla^{2j} u_k|dx = 
   L^{2j-2m}\int_{B_L\setminus B_{1/L}}|\nabla^{2j}\eta_k|dx \\
   & = L^{2j-2m}\int_{B_L\setminus B_{1/L}}|\nabla^{2j}\eta_0|dx + o(1) \geq CL^{2j} + o(1)
\end{split}
\end{equation}
for some constant $C>0$ independent of $L$.

But \eqref{eqa} is incompatible with the estimate of Lemma \ref{lemma3} when $L$ and
$k$ are large. Hence case (i) occurs (with $x_0=0$, by radial symmetry). 
In particular, we have $\int_{\R{2m}}(2m-1)!e^{2m\eta_0}dx=\Lambda_1$. 
\hfill$\square$

\section{The general case}\label{sec3}

The following gradient bound analogous to \cite{dru}, Proposition 2, 
and generalizing \cite{str}, Proposition 4.1,
will be crucial in the sequel. The proof will be given in the next
section.

\begin{prop}\label{grad} There exists a uniform constant $C$ such that
$$\sup_{x\in\Omega}\inf_{1\leq j\leq I}|x-x_k^{(j)}|^\ell u_k(x)|\nabla^\ell u_k(x)| 
\leq C \ \textrm{ for all } 1\leq \ell\leq 2m-1,\ k \in \N.$$
\end{prop}

Fix an index $i \in \{1,\dots, I\}$ 
and let $x_k=x_k^{(i)}\to x^{(i)}$, $r_k=r_k^{(i)}\to 0$ as given by 
Theorem \ref{trm1}. After a translation we may assume that $x^{(i)}=0$. Set as before
$$e_k:=\lambda_k u_k^2 e^{mu_k^2},\quad f_k:=\lambda_ku_k(0)u_ke^{mu_k^2},$$
and
$$\Lambda_k(r):=\int_{B_r}e_kdx.$$
In the following we will use the notation
$$\bar f(r):=\Intm_{\partial B_r}fd\sigma,$$
for any function $f$. Set also
$$\tilde e_k:=\lambda_k\bar u_k^2e^{m\bar u_k^2}\leq \bar e_k.$$
(Here we used Jensen's inequality.) Again we let $w_k(x):=u_k(0)(u_k(x)-u_k(0))$,
satisfying
$$(-\Delta)^m\bar w_k=\lambda_k u_k(0)\overline{u_ke^{mu_k^2}}= \bar f_k .$$
Finally set
\begin{equation}\label{sigmakbis}
\tilde \Lambda_k(r):=\int_{B_r}\tilde e_kdx\leq \Lambda_k(r),\ \sigma_k(r):=\int_{B_r}\bar f_kdx.
\end{equation}
Again Theorem \ref{trm1} implies
\begin{equation}\label{56}
\lim_{L\to\infty}\lim_{k\to\infty}\tilde\Lambda_k(Lr_k)
=\lim_{L\to\infty}\lim_{k\to\infty}\Lambda_k(Lr_k)
= \lim_{L\to\infty}\lim_{k\to\infty}\sigma_k(Lr_k)=\Lambda_1.
\end{equation}

Recalling that $x_k^{(i)}=0$ we let 
$$\rho_k=\rho_k^{(i)}:=\min\big\{\inf_{j\neq i}\frac{|x_k^{(j)}|}{2},\dist(0,\partial \Omega_k)\big\};$$ 
that is, we set $\rho_k = \dist(0,\partial \Omega_k)$ if the $(x^{(i)}_k)$ are the only
concentration points.
Observe that by Theorem \ref{trm1} we have $r_k=o(\rho_k)$ as $k \to \infty$.

Note that Proposition \ref{grad} implies the uniform bound 
\begin{equation}\label{eq57}
0\leq\sup_{r/2\le |x|\le r}u_k^2(x)-\inf_{r/2\le |x|\le r}u_k^2(x)\leq 
Cr\sup_{|x|=r}|\nabla u_k^2(x)| \leq C
\end{equation}
for $0\leq r\leq \rho_k$.

\begin{lemma}\label{lemma2ter} 
Let $0 < \ve < 1$ and assume that for $k\geq k_0=k_0(\ve)$ there holds
$$\inf_{0\leq r\leq \rho_k}\bar u_k(r)\leq \frac{\ve \bar u_k(0)}{2}.$$
Let $T_k=T_k(\ve) \le S_k=S_k(\ve)\in]0,\rho_k]$ be the smallest numbers
such that $\bar u_k(T_k)=\ve u_k(0)$, $\bar u_k(S_k)=\ve u_k(0)/2$, respectively. 
Then
\begin{equation}\label{tk}
\lim_{k\to\infty}\frac{r_k}{T_k}=\lim_{k\to\infty}\frac{T_k}{S_k}=0.
\end{equation}
Moreover for any $b<2$ and $k\geq k_0= k_0(b)$ there holds
\begin{equation}\label{eq60}
\bar w_k(r)\leq b\log\bigg(\frac{r_k}{r}\bigg)+C\ \text{ for } 0\leq r \leq T_k,
\end{equation}
and we have
\begin{equation}\label{eq61}
\lim_{k\to\infty}\tilde \Lambda_k(T_k)=\Lambda_1.
\end{equation}
\end{lemma}

\begin{proof} Property \eqref{tk} follows from \eqref{eq57} and our choice of $T_k$ and $S_k$. 

As in the proof of Lemma \ref{lemma2} for a given $t\leq T_k$ we decompose $\bar w_k=g_k+h_k$ 
on $B_t$, with 
$$\Delta^m h_k=0\textrm{ in }B_t,\ \textrm{ and } 
g_k=\partial_\nu g_k=\ldots=\partial_\nu^{m-1}g_k=0\textrm{ on }\partial B_t.$$
By \eqref{56}, we get the analogues of Lemma \ref{lemmaL} and of \eqref{stima3}; that is,
for $L\geq L_0=L_0(b)$, $k\geq k_0=k_0(L)$ there holds
$$(-1)^{m-1}t^{m-1}\underbrace{\Big(t^{-1} \Big(t^{-1}\cdots\Big(t^{-1}}_{m-1 \textrm{ times}}
\Big(\bar w_k'(t)+\frac{b}{t}\Big)\underbrace{\Big)'\cdots\Big)'\Big)'}_{m-1\textrm{ times}}
\leq 0$$
for all $t\in[Lr_k,S_k]$. We now inductively integrate from $t$ to $S_k$ as in Lemma \ref{lemma2}. 
Using Proposition \ref{grad} to bound
\begin{equation*}
|\de_r^j\bar w_k(S_k)|=\frac{u_k(0)}
{\bar u_k(S_k)}\frac{S_k^j \bar u_k(S_k)|\de_r^j \bar u_k(S_k)|}{S_k^j}
\leq \frac{C}{\ve S_k^j},
\end{equation*}
and recalling \eqref{tk}, for $L\geq L_0$ and $k\geq k_0$ we get
$$t \bar w_k'(t)\leq -b +\frac{C}{\ve}\frac{t^2}{S_k^2}= -b + o(1) 
\ \text{ for all } Lr_k\leq t\leq T_k,$$
with error $o(1) \to 0$ as $k \to \infty$.
Since $b<2$ is arbitrary, \eqref{eq60} follows as before. 

In order to prove \eqref{eq61} observe that the definition of $r_k$ gives
\begin{equation*}
   \begin{split}
      \tilde{e}_k(r) &\leq C\lambda_k u^2_k(0)e^{m u^2_k(0)}
        e^{2m\big(1+\frac{\bar w_k(r)}{2u^2_k(0)}\big)\bar w_k(r)}\\
     &\leq C\lambda_k r_k^{2m}u_k^2(0)e^{mu_k^2(0)}r_k^{-2m}e^{m(\ve+1)\bar w_k(r)}
      \leq Cr_k^{-2m}\bigg(\frac{r_k}{r}\bigg)^{m(\ve+1)b}
   \end{split}
\end{equation*}
for $Lr_k\leq r\leq T_k$. We then complete the proof as in the radial case.
\end{proof}

For $0\leq s<t\leq\rho_k$ set
$$N_k(s,t):=\Lambda_k(t)-\Lambda_k(s)=\int_{B_t\backslash B_s} \lambda_k u_k^2e^{mu_k^2}dx,$$
and let
\begin{equation}\label{eq62}
\tilde N_k(s,t):=\tilde \Lambda(t)-\tilde\Lambda(s)
=\int_s^t\omega_{2m-1}\lambda_kr^{2m-1}\bar u_k^2 e^{m\bar u_k^2}dr\leq N_k(s,t).
\end{equation}
From \eqref{eq57} we infer
\begin{equation}\label{eq58}
\sup_{|x|=r}e^{mu_k^2(x)}\leq Ce^{m\bar u_k^2(r)}\ \text{ for }0\leq r\leq \rho_k;
\end{equation}
hence we obtain
\begin{equation}\label{eq59}
\sup_{|x|=r}u_k^2(x)e^{mu_k^2(x)}\leq C(1+\bar u_k^2(r))e^{m\bar u_k^2(r)} 
\ \text{ for }0\leq r\leq \rho_k.
\end{equation}
Then \eqref{eq59} implies
\begin{equation}\label{eq63}
N_k(s,t)\leq C\tilde N_k(s,t)+o(1) \ \text{ for }0\leq s\leq t\leq \rho_k,
\end{equation}
with $o(1)\to 0$ as $k\to \infty$.
Similarly, setting
$$\tilde P_k(t)=t\int_{\de B_t}\tilde e_kd\sigma
=\omega_{2m-1}\lambda_k t^{2m}\bar u_k^2(t)e^{m\bar u_k^2(t)}
\leq P_k(t):= t\int_{\de B_t}e_kd\sigma,$$
we can estimate
\begin{equation}\label{eq63bis}
P_k(t)\leq C\tilde P_k(t)+o(1) \ \text{ for }0\leq t\leq \rho_k,
\end{equation}
with $o(1)\to 0$ as $k\to \infty$.
Finally, from \eqref{eq59} we also obtain the analogue of \eqref{pkt}; that is, we have
\begin{equation}\label{eq26bis}
P_k(t)\leq C N_k(t/2,t)+o(1)\leq CP_k(t/2)+o(1),
\end{equation}
with error $o(1)\to 0$ as $k\to\infty$.

In particular, we obtain the following improvement of Lemma \ref{lemma2ter}.
 
\begin{lemma}\label{lemma4.3} For any $0 < \ve < 1$, if $T_k=T_k(\ve)\leq \rho_k$ is as in 
Lemma \ref{lemma2ter}, then we have
$$\lim_{k\to\infty}\Lambda_k(T_k)=\Lambda_1.$$
\end{lemma}

\begin{proof}
Indeed \eqref{eq61} and \eqref{eq63} imply
$$\lim_{L\to\infty}\lim_{k\to\infty} N_k(Lr_k,T_k)
\le C \lim_{L\to\infty}\lim_{k\to\infty}\tilde N_k(Lr_k,T_k)=0,$$
which together with \eqref{56} implies the lemma.
\end{proof}

If the assumptions of Lemma \ref{lemma2ter} hold for any $0 < \ve < 1$
we may proceed to resolve secondary 
concentrations at scales $o(\rho_k)$ as in the radially symmetric case. 
Indeed, by Lemmas \ref{lemma2ter} and \ref{lemma4.3} we may then choose a subsequence 
$(u_k)$, numbers $\ve_k\to 0$ as $k\to\infty$ and corresponding numbers 
$s_k=T_k(\ve_k)\leq \rho_k$ with $r_k/s_k\to 0$ as $k\to\infty$ and such that 
$$\lim_{k\to\infty}\Lambda_k(s_k)=\Lambda_1,
\quad \lim_{L\to\infty}\lim_{k\to\infty}N_k(Lr_k,s_k)=0,$$
while in addition $\bar u_k(s_k) \to \infty$ and
$$\lim_{k\to\infty}\frac{\bar u_k(s_k)}{\bar u_k(Lr_k)}=0\ \text{ for every }L>0.$$

As before, by slight abuse of notation, we set $r_k=r_k^{(1)}$, $s_k=s_k^{(1)}$, so that
the analogue of $(H_1)$ holds, and iterate.
Suppose that for some integer $\ell\geq 1$ we already have determined numbers
$$s_k^{(0)} := 0 < r_k^{(1)}< s_k^{(1)}<\cdots< r_k^{(\ell)}< s_k^{(\ell)}=o(\rho_k)$$
satisfying the analogues of $(H_{\ell,1})$ up to $(H_{\ell,4})$. 
Similar to Lemma \ref{key} we then have the following result.

\begin{lemma}\label{keybis} There is a constant $C_0=C_0(\Lambda)$ such that for
$s_k^{(\ell)} \le t_k = o(\rho_k)$ there holds
\begin{equation}\label{poho2}
\tilde N_k(s_k^{(\ell)},t_k)\leq \frac{\tilde P_k(t_k)}{m}+C_0\tilde N_k^2(s_k^{(\ell)},t_k)+o(1),
\end{equation} 
with error $o(1)\to 0$ as $k\to\infty$.
\end{lemma}

\begin{proof}
For ease of notation we write $s=s_k^{(\ell)}$. Replacing $w_k$ with $\bar w_k$ in the proof of 
Lemma \ref{key}, similar to \eqref{Nkstima} we find
$$\tilde N_k(s,t)\leq \frac{\tilde P_k(t)}{2m}-\int_s^t\omega_{2m-1} r^{2m} 
\frac{\bar u_k(r)}{u_k(0)}\bar w_k'(r)\tilde e_kdr+o(1),$$
with error $o(1)\to 0$ as $k\to\infty$, uniformly in $s\leq t$.
Proceeding as in Lemma \ref{key}, from the equation
$$\underbrace{(t^{-1} (t^{-1}\cdots (t^{-1}}_{m-1 \textrm{ times}}\bar w_k'(t)
\underbrace{)'\cdots)')'}_{m-1\textrm{ times}}=(-1)^m\frac{A_k(t)}{\omega_{2m-1}t^{4m-3}},$$
where $A_k$ is defined by \eqref{parti6}, with $\sigma_k$ now given by \eqref{sigmakbis}, 
we get
$$t\bar w_k'(t)=-t^2\int_t^{\rho_k}t_1\int_{t_1}^{\rho_k}t_2\cdots\int_{t_{m-2}}^{\rho_k}
\frac{A_k(t_{m-1})}{\omega_{2m-1}t_{m-1}^{4m-3}}dt_{m-1}\cdots dt_1+ B_k(t,{\rho_k}),$$
where $B_k(t,\rho_k)$ corresponds to the boundary terms. 
By arguing as in the proof of Lemma \ref{lemma2} we see that $B_k$ is a linear combination 
of terms of the form
$$\frac{t^{2l+2}}{\rho_k^{2l+2}}\rho_k^j \de_r^j\bar w_k(\rho_k),\ 0 \le l \le m-2,\ 1\leq j\leq m-1.$$
After multiplication with $\frac{\bar u_k(t)}{u_k(0)}$, the resulting terms can be written as
\begin{equation*}
 \begin{split}
 \frac{t^{2l+2}}{\rho_k^{2l+2}}\bar u_k(t) \rho_k^j \de_r^j \bar u_k(\rho_k)
 &=\frac{t^{2l+1}}{\rho_k^{2l+1}}\frac{t\bar u_k(t)}{\rho_k(\bar u_k(\rho_k)+1)}
 \rho_k^j(\bar u_k(\rho_k)+1)\de_r^j\bar u_k(\rho_k).
 \end{split}
\end{equation*}
But by Proposition \ref{grad} and the analogue of \eqref{3.2} we have 
\begin{equation*}
\rho_k^j(\bar u_k(\rho_k)+1)|\de_r^j\bar u_k(\rho_k)| \le C,\ 
\frac{t\bar u_k(t)}{\rho_k(\bar u_k(\rho_k)+1)} \le C.
\end{equation*}
Hence for $t=t_k=o(\rho_k)$ we have $\frac{\bar u_k(t)}{u_k(0)}B_k(t,{\rho_k}) \to 0$ as $k \to \infty$, 
and up to an error $o(1)\to 0$ as $k\to\infty$ we obtain the identity
$$-t\bar w_k'(t)\frac{\bar u_k(t)}{u_k(0)}=t^2\frac{\bar u_k(t)}{u_k(0)}
\int_t^{\rho_k}t_1\int_{t_1}^{\rho_k}t_2\cdots\int_{t_{m-2}}^{\rho_k}
\frac{A_k(t_{m-1})}{\omega_{2m-1}t_{m-1}^{4m-3}}dt_{m-1}\cdots dt_1.$$
The rest of the proof is similar to the proof of Lemma \ref{key}.
\end{proof}

On account of \eqref{eq63} and \eqref{eq63bis} we now obtain the analogue of Lemma \ref{key2}.
The proof is the same as in the radially symmetric case.

\begin{lemma}\label{key2bis}
Let $C_0=C_0(\Lambda)$ be the constant appearing in \eqref{poho2}, and let $t_k>s_k^{(\ell)}$ 
be such that for a subsequence 
\begin{equation}\label{0alphabis}
\lim_{k\to\infty}\frac{t_k}{\rho_k}=0,\quad 0<\lim_{k\to\infty}N_k(s_k^{(\ell)},t_k)=:\alpha<\frac{1}{2C_0}.
\end{equation}
Then
$$\lim_{k\to\infty}\frac{s_k^{(\ell)}}{t_k}= 0,\
\liminf_{k\to\infty} P_k(t_k)\geq \frac{m\alpha}{2}, \text{ and }
\lim_{L\to\infty}\lim_{k\to\infty} N_k(s_k^{(\ell)},t_k/L)=0.$$
\end{lemma}

We now closely follow \cite{LRS}.
By the preceding result it suffices to consider the following two cases.
In {\bf Case A} for any sequence $t_k = o(\rho_k)$ we have
\begin{equation*}
 \sup_{s_k^{(\ell)} <t< t_k} P_k(t) \rightarrow 0 \hbox{ as } k \rightarrow \infty,
\end{equation*}
and then in view of Lemma \ref{key2bis} also
\begin{equation}\label{5.12k}
 \lim_{L\rightarrow \infty}\lim_{k\rightarrow \infty} N_k(s_k^{(\ell)},\rho_k/L) = 0,
\end{equation}
thus completing the concentration analysis at scales up to $o(\rho_k)$.

In {\bf Case B} for some $s_k^{(\ell)}<t_k\le \rho_k$ there holds
\begin{equation*}
 \limsup_{k\rightarrow \infty} N_k(s_k^{(\ell)},t_k)>0,\ \
 \lim_{k\rightarrow \infty} \frac{t_k}{\rho_k}=0.
\end{equation*}
Then, as in the radial case, from Lemma \ref{key2bis} we infer that for a
subsequence $(u_k)$ and suitable numbers $r_k^{(\ell+1)} \in ]s_k^{(\ell)},t_k[$ we have
\begin{equation}\label{5.12h}
 \lim_{k\to\infty}\frac{s_k^{(\ell)}}{r_k^{(\ell+1)}}= 0,\ 
 \lim_{k\rightarrow \infty} N_k(s_k^{(\ell)},r_k^{(\ell+1)}) > 0, \ 
 \liminf_{k\rightarrow \infty} P_k(r_k^{(\ell+1)})>0;
\end{equation}
in particular, $\bar{u}_k(r_k^{(\ell+1)}) \rightarrow \infty$ as $k\rightarrow \infty$.
Also note that
\begin{equation}\label{5.12i}
 \lim_{L\rightarrow \infty} \limsup_{k\rightarrow \infty} N_k(s_k^{(\ell)}, r_k^{(\ell+1)}/L)
 =\lim_{k\rightarrow \infty} \frac{r_k^{(\ell+1)}}{\rho_k}
 =\lim_{k\rightarrow \infty} \frac{t_k}{\rho_k}=0.
\end{equation}
Moreover, analoguous to Proposition \ref{blow} we have the following result,
which is a special case of Proposition \ref{blow0} below. 

\begin{prop} \label{prop5.1g}
There exists a subsequence $(u_k)$ such that
\begin{equation*}
 \eta^{(\ell+1)}_k(x) :=
 \bar{u}_k(r_k^{(\ell+1)})(u_k(r_k^{(\ell+1)} x ) - \bar{u}_k(r_k^{(\ell+1)}))
 \rightarrow \eta^{(\ell+1)}(x)
\end{equation*}
in $C^{2m-1}_{\loc}({\mathbb R}^{2m}\setminus \{0\})$ as $k \rightarrow \infty$,
where $\eta_0^{(\ell+1)}:=\eta^{(\ell+1)}+c^{(\ell+1)}$ solves (8), (9) for a suitable 
constant $c^{(\ell+1)}$.
\end{prop}

From Proposition \ref{prop5.1g} the desired energy quantization result at the scale
$r_k^{(\ell+1)}$ follows as in the radial case.

If $\rho_k \ge \rho_0 > 0$ we can argue as in \cite{str}, p.~416,
to obtain numbers $s_k^{(\ell+1)}$ satisfying
\begin{equation}\label{5.12a}
  \lim_{L\rightarrow \infty}
  \lim_{k\rightarrow \infty} \Lambda_k(s_k^{(\ell+1)})=(\ell+1)\Lambda_1,
\end{equation}
and such that 
\begin{equation*}
  \lim_{L\rightarrow \infty}
  \lim_{k\rightarrow \infty} (\Lambda_k(s_k^{(\ell+1)})-\Lambda_k(Lr_k^{(\ell+1)}))
  =\lim_{k\rightarrow \infty} \frac{r_k^{(\ell+1)}}{s_k^{(\ell+1)}}\nonumber\\
  =\lim_{k\rightarrow \infty}s_k^{(\ell+1)}=0,
\end{equation*}
while $\bar{u}_k(s_k^{(\ell+1)})\rightarrow \infty$ as $k \rightarrow \infty$.
Moreover, for any $L \geq 1$ we have
\begin{equation}\label{5.12b}
   \lim_{k\rightarrow \infty} \frac{\bar{u}_k(s_k^{(\ell+1)})}{\bar{u}_k(Lr_k^{(\ell+1)})}=0.
\end{equation}
By iteration we then establish \eqref{5.12a}, \eqref{5.12b} up to $\ell +1 = \ell_0$ for some
maximal index $\ell_0 \ge 1$ where Case A occurs and thus complete the concentration analysis
near the point $x^{(i)}$, getting
$$\lim_{k\to\infty}\Lambda_k(\rho_0)= \ell_0 \Lambda_1.$$

If $\rho_k\rightarrow 0$ as $k \rightarrow \infty$, we distinguish the following two cases.
In {\bf Case 1} for some $\varepsilon_0\in]0,1[$ and all $t \in [r_k^{(\ell+1)},\rho_k]$ there
holds $\bar{u}_k(t)\ge\varepsilon_0\bar{u}_k(r_k^{(\ell+1)})$. The decay estimate
that we established in Lemma \ref{lemma2ter} then remains valid throughout this range
and \eqref{5.12a} holds true for any choice $s_k^{(\ell+1)}=o(\rho_k)$.
Again the concentration analysis at scales up to $o(\rho_k)$ is complete.
In {\bf Case 2}, for any $\varepsilon\in]0,1[$ there is a minimal
$T_k=T_k(\varepsilon)\in [r_k^{(\ell+1)},\rho_k]$ as in Lemma \ref{lemma2ter} such that
$\bar{u}_k(T_k)=\varepsilon\bar{u}_k(r_k^{(\ell+1)})$.
Then as before we can define numbers $s_k^{(\ell+1)} < \rho_k$ with
$\bar{u}_k(s_k^{(\ell+1)})\rightarrow \infty$ as $k \rightarrow \infty$
so that \eqref{5.12a}, \eqref{5.12b} also hold true, and we proceed
by iteration up to some maximal index $\ell_0 \ge 1$ where either Case 1
or Case A holds with final radii $r_k^{(\ell_0)}$, $s_k^{(\ell_0)}$, respectively.

For the concentration analysis at the scale $\rho_k$
first assume that for some number $L\ge 1$ there is a sequence $(x_k)$ such that
$\rho_k/L \le R_k(x_k)\le |x_k|\le L\rho_k$ and
\begin{equation}\label{sub}
 \lambda_k |x_k|^{2m} u_k^2(x_k)e^{mu_k^2(x_k)}\ge \nu_0>0.
\end{equation}
By Proposition \ref{grad} we may assume that $|x_k|=\rho_k$. Moreover, \eqref{eq57} implies that
$\dist(0,\de\Omega_k)/\rho_k \to\infty$ as $k\to\infty$.
As in \cite{str}, Lemma 4.6, we then have 
$\bar{u}_k(\rho_k)/\bar{u}_k(r_k^{(\ell_0)}) \rightarrow 0$ as $k\rightarrow \infty$,
ruling out Case 1; that is, at scales up to $o(\rho_k)$ we end with Case A.
The desired quantization result at the scale $\rho_k$ then is a consequence of the
following result similar to \cite{str}, Proposition 4.7, whose proof may be easily carried 
over to the present situation.

\begin{prop} \label{blow0}
Assuming \eqref{sub}, there exist a finite set $S_0\subset \mathbb{R}^{2m}$ and a subsequence
$(u_k)$ such that
\begin{equation*}
 \eta_k(x) :=
 u_k(x_k)(u_k(\rho_k x ) - u_k(x_k)) \rightarrow \eta(x)
\end{equation*}
in $C^{2m-1}_{\loc}({\mathbb R}^{2m}\setminus S_0)$ as $k \rightarrow \infty$, where
for a suitable constant $c_0$ the function $\eta_0=\eta+c_0$ solves \eqref{eqliou}, \eqref{etak2}.
\end{prop}

By Proposition \ref{blow0} in case of \eqref{sub} there holds
\begin{equation}\label{3.0}
   \lim_{L\rightarrow \infty} \lim_{k\rightarrow \infty}
   \int_{\{x\in \Omega; \frac{\rho_k}{L}\le R_k(x) \le |x| \le L\rho_k\}} e_k dx
   =\Lambda_1.
\end{equation}
Letting
\begin{equation*}
  X_{k,1}=X_{k,1}^{(i)}=\{ x_k^{(j)}; \exists C>0: |x_k^{(j)}|\le C\rho_k\
  \ \text{for all} \ \ k\}
\end{equation*}
and carrying out the above blow-up analysis up to scales of order $o(\rho_k)$
also on all balls of center $x_k^{(j)}\in X_{k,1}$, then from \eqref{5.12b} and \eqref{3.0} we have
\begin{equation*}
  \lim_{L\rightarrow \infty}\lim_{k\rightarrow \infty}
  \Lambda_k (L\rho_k)= \Lambda_1(1+I_1),
\end{equation*}
where $I_1$ is the total number of bubbles concentrating at the points
$x_k^{(j)}\in X_{k,1}^{(i)}$ at scales $o(\rho_k)$.

On the other hand, if \eqref{sub} fails to hold clearly we have
\begin{equation}\label{sub1}
   \lim_{L\rightarrow \infty} \limsup_{k\rightarrow \infty}
   \int_{\{x\in \Omega; \frac{\rho_k}{L}\le R_k(x) \le |x| \le L\rho_k\}} e_k dx=0,
\end{equation}
and the energy estimate at the scale $\rho_k$ again is complete.

In order to deal with secondary concentrations around $x_k^{(i)}=0$ at scales
exceeding $\rho_k$, with $X_{k,1}$ defined as above we let
\begin{equation*}
  \rho_{k,1}=\rho_{k,1}^{(i)}=\min\big\{\inf_{\{j; x_k^{(j)}\notin X_{k,1}\}}\frac{|x_k^{(j)}|}{2},\ 
  \dist(0,\partial \Omega_k)\big\};
\end{equation*}
that is, we again set $\rho_{k,1}= \dist(0,\partial \Omega_k)$,
if $\{j; x_k^{(j)} \notin X_{k,1}\}=\emptyset$. From this definition
it follows that $\rho_{k,1}/\rho_k\rightarrow \infty$ as $k\rightarrow \infty$.
Then, using the obvious analogue of Lemma \ref{key2bis}, either we have
\begin{equation*}
  \lim_{L\rightarrow \infty} \limsup_{k\rightarrow \infty}
  N_k\Big(L\rho_k,\frac{\rho_{k,1}}{L}\Big)=0,
\end{equation*}
and we iterate to the next scale; or there exist radii $t_k \le \rho_{k,1}$ such that
$t_k/\rho_k\rightarrow \infty$, $t_k/\rho_{k,1}\rightarrow 0$
as $k\rightarrow \infty$ and a subsequence $(u_k)$ such that
\begin{equation}\label{sub1a}
  P_k(t_k)\ge \nu_0>0\ \text{ for all } k.
\end{equation}
The argument then depends on whether \eqref{sub} or \eqref{sub1} holds.
In case of \eqref{sub}, as in \cite{str}, Lemma 4.6, the bound \eqref{sub1a} and
Proposition \ref{blow} imply that $\bar{u}_k(t_k)/\bar{u}_k(\rho_k)\rightarrow 0$
as $k\rightarrow 0$. Then we can argue as in Case A for
$r\in [L\rho_k,\rho_{k,1}]$ for sufficiently large $L$, and we can continue as before
to resolve concentrations in this range of scales.

In case of \eqref{sub1} we further need to distinguish whether Case A or Case 1 holds
at the final stage of our analysis at scales $o(\rho_k)$.
In fact, for the following estimates we also consider all
points $x_k^{(j)}\in X_{k,1}^{(i)}$ in place of $x_k^{(i)}$.
Recalling that in Case A we have \eqref{5.12b} (with index $\ell_0$ instead of $\ell+1$) 
and \eqref{5.12k}, on account of \eqref{sub1}
for a suitable sequence of numbers $s_{k,1}^{(0)}$ such that
$s_{k,1}^{(0)}/\rho_k\rightarrow \infty$, $t_k/s_{k,1}^{(0)}\rightarrow \infty$ as
$k\rightarrow \infty$ we find
\begin{equation*}
 \lim_{L\rightarrow \infty} \lim_{k\rightarrow \infty}
 \Big(\Lambda(s_{k,1}^{(0)})-\sum_{x_k^{(j)}\in X_{k,1}^{(i)}}
 \Lambda_k^{(j)}(Lr_k^{(\ell_0^{(j)})})\Big)=0,
\end{equation*}
where $\Lambda_k^{(j)}(r)$ and $r_k^{(\ell_0^{(j)})}$ are computed as above with respect
to the concentration point $x_k^{(j)}$.
In particular, with such a choice of $s_{k,1}^{(0)}$ we find the
intermediate quantization result
\begin{equation*}
  \lim_{k\rightarrow \infty} \Lambda_k (s_{k,1}^{(0)})=\Lambda_1 I_1
\end{equation*}
analogous to \eqref{5.12a}, where $I_1$ is defined as above.
In Case 1 we can obtain the same conclusion by our earlier reasoning.
Moreover, in Case 1 we can argue as in \cite{str}, Lemma 4.8, to conclude that 
$\bar{u}_k(t_k)/\bar{u}_k(Lr_k^{(\ell_0^{(j)})})\rightarrow 0$ for any $L\geq 1$
as $k\rightarrow 0$; therefore, similar to \eqref{5.12b} in Case A, 
we can achieve that for any $L\geq 1$ we have
\begin{equation*}
 \lim_{k\rightarrow \infty} \frac{\bar{u}_k(s_{k,1}^{(0)})}
      {\bar{u}_k(Lr_k^{(\ell_0^{(j)})})}
 =\lim_{k\rightarrow \infty} \frac{r_k^{(\ell_0^{(j)})}}{s_{k,1}^{(0)}}
 =\lim_{k\rightarrow \infty} \frac{\rho_k}{s_{k,1}^{(0)}}
 =\lim_{k\rightarrow \infty} \frac{s_{k,1}^{(0)}}{t_k}=0
\end{equation*}
for all $x_k^{(j)}\in X_{k,1}^{(i)}$ where Case 1 holds, similar to $(H_{\ell})$.

We then finish the argument by iteration. For $\ell\ge 2$ we inductively define the sets
\begin{equation*}
 X_{k,\ell}=X_{k,\ell}^{(i)}=\{ x_k^{(j)}; \exists C>0: |x_k^{(j)}|\le C\rho_{k,\ell-1}\
 \text{ for all } k\}
\end{equation*}
and we let
\begin{equation*}
  \rho_{k,\ell}=\rho_{k,\ell}^{(i)}=\min\big\{\inf_{\{j; x_k^{(j)}\notin X_{k,1}\}}\frac{|x_k^{(j)}|}{2},\ 
  \dist(0,\partial \Omega_k)\big\};
\end{equation*}
that is, as before, we set $\rho_{k,\ell}= \dist(0,\partial \Omega_k)$, if
$\{j; x_k^{(j)}\notin X_{k,\ell}^{(i)}\}=\emptyset$. Iteratively performing
the above analysis
at all scales $\rho_{k,\ell}$, thereby exhausting all concentration points $x_k^{(j)}$,
upon passing to further subsequences, we finish the proof of Theorem \ref{trm2}.

\subsection{Proof of Proposition \ref{grad}}

Our proof of Proposition \ref{grad} is modelled on the proof of \cite{dru}, Proposition 2.
In fact, the first steps of the proof seem almost identical to the corresponding arguments
in \cite{dru}. The special character of the present problem only enters at the last stage, 
where we also need to distinguish the cases $\ell=1$ and $2\leq \ell\leq 2m-1$.

Fix any index $1 \le \ell \le 2m-1$. The following constructions will depend on this 
choice; however, for ease of notation we suppress the index $\ell$ in the sequel.

Set $R_k(x):=\inf_{1\leq j\leq I} |x-x_k^{(j)}|$ and choose points $y_k$
such that
$$R_k^{\ell}(y_k)u_k(y_k)|\nabla^{\ell} u_k (y_k)|
=\sup_{\Omega} R_k^{\ell} u_k|\nabla^{\ell} u_k|=:L_k.$$
Suppose by contradiction that $L_k\to\infty$ as $k\to\infty$. From Theorem \ref{trm1} 
then it follows that $s_k:=R_k(y_k)\to 0$ as $k\to\infty$. Set
$$\Omega_k:=\{y; y_k+s_k y\in\Omega\}$$
and let
$$v_k(y):= u_k(y_k+s_ky),\quad y\in\Omega_k.$$
Observe that for $1\le j \le m$ via Sobolev's embedding from \eqref{eq1} we obtain
\begin{equation}\label{eq72}
\|\nabla^j v_k\|^2_{L^{\frac{2m}{j}}(\Omega_k)}
\leq C\|\nabla^m v_k\|^2_{L^2(\Omega_k)}
= C \int_{\Omega_k} v_k(-\Delta)^m v_k dx\leq C.
\end{equation}
Also let
$$y_k^{(i)}:=\frac{x_k^{(i)}-y_k}{s_k},\ 1\leq i\leq I$$
and set 
$$S_k:=\{y_k^{(i)};1\leq i\leq I\}.$$

Clearly then we have
$$\dist(0,S_k)=\inf_{1\leq i\leq I}|y_k^{(i)}|=1$$
and
\begin{equation}\label{eq73}
\sup_{y\in\Omega_k}( \dist(y,S_k)^{\ell}v_k(y)|\nabla^{\ell} v_k(y)|)
= v_k(0)|\nabla^{\ell} v_k(0)|=L_k\to\infty
\end{equation}
as $k\to\infty$. Moreover \eqref{eq9} implies 
\begin{equation}\label{eq74}
0\leq v_k(-\Delta)^m v_k=\lambda_k s_k^{2m}v_k^2e^{mv_k^2}\leq \frac{C}{\dist(y,S_k)^{2m}}.
\end{equation}

Since $\lim_{k\to \infty }s_k=0$, we may assume that as $k\to\infty$ the domains $\Omega_k$ 
exhaust a half-space
$$\Omega_0=\R{2m-1}\times]-\infty,R_0[, $$
where $0<R_0\leq\infty$. 
We may also assume that either $\lim_{k\to\infty} |y_k^{(i)}|=\infty$ or 
$\lim_{k\to\infty} y_k^{(i)}=y^{(i)}$, $1\leq i\leq I$, and we let $S_0$ be the set of these
accumulation points of $S_k$, satisfying $\dist(0,S_0)=1$. For $R>0$ denote 
$$K_{k,R}:= \Omega_k\cap B_R(0)\backslash \bigcup_{y\in S_0}\overline{B_{1/R}(y)}.$$ 
Observing that $\lambda_k s_k^{2m}\to 0$, from \eqref{eq74} we obtain that
\begin{equation}\label{eq75}
\lim_{k\to\infty}\|\Delta^m v_k\|_{L^\infty(K_{k,R})}=0 \quad\text{for every }R>0.
\end{equation}

\begin{lemma}\label{lemma16} We have $R_0=\infty$, hence $\Omega_0=\R{2m}$.
\end{lemma}

\begin{proof} Suppose by contradiction that $R_0<\infty$. Choosing $R=2R_0$ and observing 
that by \eqref{1.2} for $0 \le j < \ell \le 2m-1$ we have $\partial_{\nu}^jv_k^2=0$ 
on $\partial \Omega_k$, from Taylor's formula and \eqref{eq73} we conclude
$$\sup_{K_{k,R}}\frac{v_k^2}{v_k(0)|\nabla^{\ell} v_k(0)|}\leq C=C(R).$$
Letting $w_k:=\frac{v_k}{\sqrt{v_k(0)|\nabla v_k(0)|}}$, we then have $0\leq w_k\leq C$ on $K_{k,R}$. 
Using \eqref{eq72}, Sobolev's embedding, \eqref{eq73} and \eqref{eq75} we infer
$$\|\nabla w_k\|_{L^{2m}(\Omega_k)}+\|\nabla^2 w_k\|_{L^m(\Omega_k)}
+\|\Delta^m w_k\|_{L^\infty(K_{k,R})}\to 0\ \text{ as }k\to\infty.$$
Since $\de_\nu^jw_k=0$ on $\de \Omega_k$ for $0\leq j\leq m-1$, it follows from elliptic 
regularity that $w_k\to 0$ in $C^{2m-1,\alpha}_{\loc}(K_{k,R})$ for $0<\alpha<1$, 
contradicting the fact that $w_k(0)|\nabla^{\ell} w_k(0)|=1$.
\end{proof}

\begin{lemma}\label{lemma5.2} As $k\to\infty$ we have $v_k(0)\to\infty$ and
$$\frac{v_k}{v_k(0)}\to 1\ \text{ in } C^{2m-1,\alpha}_{\loc}(\R{2m}\backslash S_0).$$
\end{lemma}

\begin{proof} First observe that
$$c_k:=\sup_{B_{1/2}}v_k \to \infty\quad \text{as }k\to\infty.$$
Indeed, otherwise \eqref{eq72}, \eqref{eq75} and elliptic regularity would contradict \eqref{eq73}.
Letting $w_k:=\frac{v_k}{c_k}$, from \eqref{eq72} and \eqref{eq75} for any $R>0$ we have
$$\|\nabla w_k\|_{L^{2m}(\Omega_k)}+\|\nabla^2 w_k\|_{L^{m}(\Omega_k)}
+\|\Delta^m w_k\|_{L^\infty (K_{k,R})} \to 0 \quad \text{as }k\to\infty,$$
whence $w_k\to w\equiv const$ in $C^{2m-1,\alpha}_{\loc}(\R{2m}\backslash S_0)$. 
Recalling that $\dist(0,S_0)=1$, we obtain
$$w\equiv \sup_{B_{1/2}} w=\lim_{k\to\infty}\sup_{B_{1/2}} w_k=1.$$
In particular we conclude that $\frac{v_k(0)}{c_k}=w_k(0)\to 1$ as $k\to\infty$ and 
therefore $v_k(0)=c_k w_k(0)\to\infty$, $\frac{v_k}{v_k(0)}=\frac{w_k}{w_k(0)}\to 1$ 
in $C^{2m-1,\alpha}_{\loc}(\R{2m}\backslash S_0)$, as claimed.
\end{proof}

For the final argument now we need to distinguish the cases $\ell=1$ and $2\leq \ell\leq 2m-1$.
Consider first the case $\ell=1$. Set
$$\tilde v_k(y):=\frac{v_k(y)-v_k(0)}{|\nabla v_k(0)|}.$$
From \eqref{eq73} and Lemma \ref{lemma5.2} we infer
\begin{equation}\label{eq76}
|\nabla \tilde v_k(y)|=\frac{v_k(0)}{v_k(y)}\frac{v_k(y)|\nabla v_k(y)|}{v_k(0)|\nabla v_k(0)|}
\leq\frac{1+o(1)}{\dist (y,S_0)},
\end{equation}
with error $o(1)\to 0$ in $C^{2m-1,\alpha}_{\loc}(\R{2m}\backslash S_0)$ as $k\to\infty$. 
Since $\tilde v_k(0)=0$, from \eqref{eq76} we conclude that $\tilde v_k$ is bounded in 
$C^1(K_{k,R})$ for every $R>0$, uniformly in $k$. 
Moreover, \eqref{eq74} and Lemma \ref{lemma5.2} give
\begin{equation}\label{eq77}
|\Delta^m \tilde v_k|=\frac{v_k(0)}{v_k}\frac{v_k|\Delta^{m}v_k|}{v_k(0)|\nabla v_k(0)|}
\leq C(R)\frac{v_k(0)}{L_kv_k}\to 0
\end{equation}
uniformly on $K_{k,R}$ as $k\to\infty$, for any $R>0$. The sequence $\tilde v_k$ then is bounded 
in $C^{2m-1,\alpha}_{\loc}(\R{2m}\backslash S_0)$ for any $\alpha < 1$, and by Arzel\`a-Ascoli's 
theorem we can assume that $\tilde v_k\to \tilde v$ in $C^{2m-1,\alpha}_{\loc}(\R{2m}\backslash S_0)$, 
where $\tilde v$ satisfies
\begin{equation}\label{eq78}
\Delta^m \tilde v=0,\quad \tilde v(0)=0,\quad |\nabla \tilde v(0)|=1,\quad |\nabla \tilde v(y)|\leq \frac{1}{\dist(y,S_0)}.
\end{equation}

Fix a point $x_0\in S_0$. For any $r\in ]0, \dist(x_0, S_0\setminus \{x_0\})/2[$ 
let $\varphi \in C^{\infty}_0(B_r(x_0))$ be a function $0 \le \varphi \le 1$ 
such that $\varphi \equiv 1$ in $B_{r/2}(x_0)$, and satisfying 
$|\nabla^j\varphi| \le Cr^{-j}$ for $0 \le j \le m$. Integration by parts yields
\begin{equation}\label{5.1}
 \begin{split}
   \int_{B_r(x_0)} (\nabla \varphi v_k\cdot \nabla \Delta^{m-1}v_k & + \varphi v_k\Delta^m v_k)dx\\
   & = - \int_{B_r(x_0)}\varphi \nabla  v_k \cdot \nabla \Delta^{m-1} v_k\; dx =: I.
 \end{split}
\end{equation}
Again integrating by parts $m-1$ times, we obtain
\begin{equation*}
  I = (-1)^{m} \int_{B_r(x_0)}\sum_{|\alpha| = m-1}\partial^{\alpha}(\varphi \nabla  v_k)
  \cdot \nabla\partial^{\alpha} v_k\; dx,
\end{equation*}
so that by H\"older's inequality and \eqref{eq72} this term may be bounded  
\begin{equation*}
 \begin{split}
  |I| & \le C \sum_{1\le j \le m} r^{j-m} \int_{B_r(x_0)}|\nabla^{j}v_k||\nabla^{m} v_k| dx \\
  & \le C\sum_{1\le j \le m}\|\nabla^{j}v_k\|_{L^{\frac{2m}{j}}}\|\nabla^{m} v_k\| _{L^2}
  \le C.
 \end{split}
\end{equation*}
Similarly, we have 
\begin{equation*}
   0 \le \int_{B_r(x_0)} \varphi v_k(-\Delta)^m v_k \; dx \le C,
\end{equation*}
and from \eqref{5.1} we conclude the bound 
\begin{equation}\label{5.2}
 \begin{split}
   \Big|\int_{B_r(x_0)} & \nabla \varphi v_k\cdot \nabla \Delta^{m-1}v_k dx\Big| \le C.
 \end{split}
\end{equation}
Observe that $\nabla \varphi = 0$ in $B_{r/2}(x_0)$. By Lemma \ref{lemma5.2} therefore
the integral on the left-hand side equals
\begin{equation*}
 \begin{split}
   \int_{B_r(x_0)} & \nabla \varphi v_k\cdot\nabla \Delta^{m-1}v_k dx \\
   & = (1+o(1))v_k(0)|\nabla v_k(0)|\int_{B_r(x_0)}  \nabla \varphi \cdot\nabla \Delta^{m-1}\tilde v_k dx\\
   & = -(1+o(1))v_k(0)|\nabla v_k(0)|\int_{B_r(x_0)}  \varphi \Delta^{m}\tilde v_k dx.
 \end{split}
\end{equation*}
Since $(-\Delta)^{m}\tilde v_k\geq 0$, it follows that
$$\int_{B_{r/2}(x_0)}(-\Delta)^{m}\tilde v_k dx \le \frac{C}{v_k(0)|\nabla v_k(0)|} 
= CL_k^{-1} \to 0 \ \textrm{ as } k\to\infty.$$
Recalling  \eqref{eq77}, we infer that $\Delta^m \tilde v_k\to 0$ in $L^1_{\loc}(\R{2m})$.
Therefore $\Delta^m \tilde v\equiv 0$ in $\R{2m}$. Since from \eqref{eq78} we have
$|\tilde v(y)|\leq C(1+|y|)$ for $y\in\R{2m}$, we may now invoke a Liouville-type theorem as
in \cite{mar1}, Theorem 5, to see that $\tilde v$ is a polynomial of degree at most $2m-2$
if $m>1$ and of degree at most $1$ if $m=1$. But then \eqref{eq78} implies that
$\tilde v\equiv 0$, contradicting the fact that $|\nabla \tilde v_k(0)|=1$. This completes
the proof in the case $\ell=1$.

In the case when $2\leq \ell\leq m-1$ we set
$$\tilde v_k(y):=\frac{v_k(y)-v_k(0)}{|\nabla^\ell v_k(0)|}.$$
As shown above we have
$$\dist(y,S_k)v_k(y)|\nabla v_k(y)|\leq C\sup_{x\in \Omega}R_k(x)u_k(x)|\nabla u_k(x)|\leq C;$$
hence Lemma \ref{lemma5.2} implies with error $o(1)\to 0$ in 
$C^{2m-1,\alpha}_{\loc}(\R{2m}\backslash S_0)$ as $k \to \infty$ that
\begin{equation}\label{eq76bis}
|\nabla \tilde v_k|\leq \frac{C(1+o(1))}{v_k(0)|\nabla^\ell v_k(0)|\dist(y,S_0)}
=\frac{C(1+o(1))}{L_k\dist(y,S_0)} \to 0.
\end{equation}
Notice that this is stronger than its analogue \eqref{eq76}.
As in the case $\ell=1$ we have
\begin{equation}\label{eq77bis}
\Delta^m \tilde v_k=\frac{v_k(0)}{v_k}\frac{v_k\Delta^{m}v_k}{v_k(0)|\nabla^\ell v_k(0)|}
\leq \frac{C(R)}{L_k}\to 0
\end{equation}
uniformly on $K_{k,R}$ as $k\to\infty$, for any $R>0$, hence $\tilde v_k\to \tilde v$ 
in $C^{2m-1,\alpha}_{\loc}(\R{2m}\backslash S_0)$, where $\tilde v$ satisfies
\begin{equation*}
\Delta^m \tilde v=0,\quad \tilde v(0)=0,\quad |\nabla^\ell \tilde v(0)|=1.
\end{equation*}
On the other hand \eqref{eq76bis} implies $\nabla \tilde v\equiv 0$, contradiction. 
This completes the proof. \hfill $\square$

\appendix

\section*{Appendix}\label{appendix}

We collect here some technical results used in the above sections.
The proof of the following proposition can be found in \cite{mar1}, Prop. 4.
\begin{prop}\label{c2m0} Let $\Delta^{m}h=0$ in $B_{2}\subset\R{n}$. For every $0\leq \alpha<1$ and $\ell\geq 0$
there is a  constant $C(\ell,\alpha)$ independent of $h$ such that
$$\|h\|_{C^{\ell,\alpha}(B_1)}\leq C(\ell,\alpha)  \|h\|_{L^1(B_2)}.$$
\end{prop}

By a simple covering argument Proposition \ref{c2m0} can be extended to the case of annuli.

\begin{prop}\label{c2m} Let $\Delta^{m}h=0$ in $B_{2L}(0)\backslash B_{1/2L}(0)\subset\R{n}$ for some $L\geq 1$. 
For every $0\leq \alpha<1$ and $\ell\geq 0$ there is a  constant $C=C(\ell,\alpha,L)$ such that
$$\|h\|_{C^{\ell,\alpha}(B_{L}(0)\backslash B_{1/L}(0))}
\leq C \|h\|_{L^1(B_{2L}(0)\backslash B_{1/2L}(0))}.$$
\end{prop}

\begin{lemma}\label{lemmaparti} Let $g\in C^\infty(\overline B_t)$, where $B_t=B_t(0) \subset \R{n}$ 
for some $n\in \N$, $t>0$. Assume that $g$ is radially symmetric and satisfies
\begin{equation}\label{bordo}
g=\partial_\nu g=\ldots=\partial_\nu^{m-1}g=0\quad \textrm{on }\partial B_t.
\end{equation}
Then
\begin{equation}\label{parti}
\int_{\partial B_t}t^{m-1}\de_\nu^m g\, d\sigma
=\int_0^{t}t_2\cdots\int_0^{t_{m-1}}t_m\bigg(\int_{B_{t_m}}\Delta^{m}gdx\bigg) dt_m\ldots dt_2.
\end{equation}
\end{lemma}

\begin{proof} For $m=1$ equation \eqref{parti} simply reduces to
\begin{equation}\label{a.1}
\int_{\partial B_t}\de_\nu g\, d\sigma=\int_{B_t}\Delta g\, dx.
\end{equation}
For $m = 2$ consider the function $\varphi(x)= x\cdot \nabla g(x)$ with 
$$\int_{\partial B_t}\de_\nu \varphi\, d\sigma
=\int_{B_t}\Delta \varphi\, dx=\int_{B_t}\Delta (x\cdot \nabla g) dx
=\int_{B_t}(x\cdot \nabla\Delta g + 2\Delta g)dx$$
and note that the condition $\partial_\nu g=0$ on $\de B_t$ and \eqref{a.1} imply 
$$\int_{\partial B_t}\de_\nu \varphi\, d\sigma 
= \int_{\partial B_t}\de_\nu (x\cdot \nabla g) d\sigma 
= \int_{\partial B_t}t\de_\nu^2 g\, d\sigma,\ \text{ and } 
\int_{B_t}\Delta g\, dx = 0.$$
Thus from Fubini's theorem we obtain the desired identity
\begin{equation*}
\begin{split}
  \int_{\partial B_t}t\de_\nu^2 g\, d\sigma & = \int_{B_t}x\cdot \nabla\Delta g\, dx\\
  & =\int_0^{t}t_2\bigg(\int_{\partial B_{t_2}}\de_\nu \Delta g\, d\sigma\bigg) dt_2
  =\int_0^{t}t_2\bigg(\int_{B_{t_2}}\Delta^2 g\, dx\bigg) dt_2.
\end{split}
\end{equation*}
We now proceed by induction. Assume that the lemma is true for $m-1$. 
Choosing $\varphi(x)= x\cdot \nabla g(x)$ with
$$\varphi=\partial_\nu \varphi=\ldots=\partial_\nu^{m-2}\varphi=0\quad \textrm{on }\partial B_t$$ we get
\begin{equation*} 
\begin{split}
& \int_{\partial B_t} t^{m-1} \de_\nu^m gd\sigma=\int_{\partial B_t}t^{m-2}\de_\nu^{m-1}(t\de_\nu g)d\sigma
=\int_{\de B_t}t^{m-2}\de_\nu^{m-1}(x\cdot \nabla g)d\sigma\\
&\quad =\int_0^t t_2\cdots\int_0^{t_{m-2}}t_{m-1}
\int_{B_{t_{m-1}}}\Delta^{m-1}(x\cdot \nabla g)dx dt_{m-1}\ldots dt_2=:I.
\end{split}
\end{equation*}
Observe that 
$\Delta^{m-1}(x\cdot \nabla g)=x\cdot \nabla \Delta^{m-1}g+2(m-1)\Delta^{m-1}g$, hence
\begin{equation*} 
\begin{split}
I&=\int_0^t t_2\cdots\int_0^{t_{m-2}}t_{m-1}
\int_{B_{t_{m-1}}} (x\cdot \nabla \Delta^{m-1} g) dx dt_{m-1}\ldots dt_2\\
&\quad +2(m-1)\int_0^t t_2\cdots\int_0^{t_{m-2}}t_{m-1}\int_{B_{t_{m-1}}}\Delta^{m-1}g
dx dt_{m-1}\ldots dt_2\\
&=II+III.
\end{split}
\end{equation*}
By inductive hypothesis and \eqref{bordo} the contribution from the second term is
\begin{equation*}
III=2(m-1)\int_{\partial B_t}t^{m-2}\partial_\nu^{m-1}gd\sigma=0,
\end{equation*}
and our claim follows from writing
\begin{equation}\label{parti3}
\begin{split}
\int_{B_{t_{m-1}}}x\cdot \nabla\Delta ^{m-1}g dx 
& =\int_0^{t_{m-1}}t_m\int_{\partial B_{t_m}} \de_\nu \Delta^{m-1}g d\sigma dt_m\\
& =\int_0^{t_{m-1}}t_m\int_{B_{t_m}} \Delta^m g dx dt_m.
\end{split}
\end{equation}
\end{proof}

\begin{lemma}\label{trmliou} Let $u\in C^\infty(\R{n})\cap L^p(\R{n})$, for some $p\geq 1$, satisfy $\Delta^j u=0$ for some integer $j>0$. Then $u\equiv 0$.
\end{lemma}

\begin{proof} We first claim that
$$\lim_{R\to\infty}\Intm_{B_R(\xi)}u dx=0$$
for every $\xi\in\R{n}$. Indeed by Jensen's inequality
$$\bigg|\Intm_{B_R(\xi)}u dx\bigg|\leq \Intm_{B_R(\xi)}|u| dx\leq
\bigg(\Intm_{B_R(\xi)}|u|^p dx\bigg)^\frac{1}{p}\leq \frac{1}{R^{n/p}}\|u\|_{L^p(\R{n})}\to 0,$$
as $R\to\infty$. By Pizzetti's formula (see \cite{piz}) we have constants $c_1,\ldots, c_{j-1}$ such that
$$\Intm_{B_R(\xi)}u dx=u(\xi)+c_1R^2\Delta u(\xi)+\cdots+c_{j-1}R^{2j-2}\Delta^{j-1}u(\xi)=:P(R).$$
Taking the limit as $R\to\infty$ we see at once that the polynomial $P(R)$ is identically $0$, and in particular $u(\xi)=P(0)=0$. Since $\xi$ was arbitrary the proof is complete.
\end{proof}

\begin{lemma}\label{lemcap} There holds
$$\mathrm{cap}_{H^m}(\{0\}) = \inf \{\|\nabla^m \varphi\|_{L^2};\ \varphi \in X\} = 0,$$
where
$$X=\{\varphi \in C^{\infty}_0(B_1(0));\ 0 \le \varphi \le 1,\exists r >0: \varphi(x)= 1 \text{ for } |x| \le r \}.$$
\end{lemma}

\begin{proof} Let $f(x) = \log \log \log(1/|x|)$ 
with $\nabla^m f \in L^2(B_{e^{-e}}(0))$ and fix $g \in C^{\infty}({\mathbb R})$ 
with $0 \le g \le 1$ satisfying $g(s) = 0$ for $s \le 0$,
$g(s) = 1$ for $s \ge 1$. Letting 
$$\varphi_k(x) = g(f(x)-k),\ k \in \N,$$
we find $\varphi_k \in X$ for all $k$ and
$\|\nabla^m \varphi_k\|_{L^2} \to 0$ as $k \to \infty$.
\end{proof}



\begin{thebibliography}{2}
\bibitem{AD} \textsc{Adimurthi, O. Druet,} \emph{Blow-up analysis in dimension $2$ and a sharp form of Trudinger-Moser inequality}, Comm. Partial Differential Equations \textbf{29} (2004), 295-322.
\bibitem{ARS} \textsc{Adimurthi, F. Robert, M. Struwe,} \emph{Concentration phenomena for Liouville's equation in dimension 4}, J. Eur. Math. Soc. \textbf{8} (2006), 171-180.
\bibitem{AS} \textsc{Adimurthi, M. Struwe,} \emph{Global compactness properties of semilinear elliptic equations with critical exponential growth}, J. Functional Analysis \textbf{175} (2000), 125-167.
\bibitem{cha} \textsc{S-Y. A. Chang} \emph{Non-linear Elliptic Equations in Conformal Geometry}, 
Zurich lecture notes in advanced mathematics, EMS (2004).
\bibitem{dru} \textsc{O. Druet,} \emph{Multibumps analysis in dimension $2$: quantification of blow-up levels}, Duke Math. J. \textbf{132} (2006), 217-269.
\bibitem{LRS} \textsc{T. Lamm, F. Robert, M. Struwe} \emph{The heat flow with a critical exponential nonlinearity}, J. Functional Anal. \textbf{257} (2009) 2951-2998.
\bibitem{mar1} \textsc{L. Martinazzi} \emph{Classification of solutions to the higher order Liouville's equation in $\R{2m}$}, Math. Z. \textbf{263} (2009), 307-329.
\bibitem{mar2} \textsc{L. Martinazzi} \emph{A threshold phenomenon for embeddings of $H^m_0$ into Orlicz
spaces}, Calc. Var. Partial Differential Equations \textbf{36} (2009), 493-506.
\bibitem{MP} \textsc{L. Martinazzi, M. Petrache} \emph{Asymptotics and quantization for a mean-field equation of higher order}, Comm. Partial Differential Equations \textbf{35} (2010), 1-22.
\bibitem{piz} \textsc{P. Pizzetti} \emph{Sulla media dei valori che una funzione dei punti dello spazio assume alla superficie di una sfera}, Rend. Lincei \textbf{18} (1909), 182-185.
\bibitem{RS} \textsc{F. Robert, M. Struwe} \emph{Asymptotic profile for a fourth order PDE with critical exponential growth in dimension four}, Adv. Nonlin. Stud. \textbf{4} (2004), 397-415.
\bibitem{RW} \textsc{F. Robert, J.-C. Wei} \emph{Asymptotic behavior of a fourth order mean field equation with Dirichlet boundary condition}, Indiana Univ. Math. J. \textbf{57} (2008), 2039-2060.
\bibitem{str} \textsc{M. Struwe} \emph{Quantization for a fourth order equation with critical exponential growth}, Math. Z. \textbf{256} (2007), 397-424.
\bibitem{wei} \textsc{J.-C. Wei} \emph{Asymptotic behavior of a nonlinear fourth order eigenvalue problem}, Comm. Partial Differential Equations \textbf{21} (1996), 1451-1467.
\end{thebibliography}
\end{document}